\documentclass[12pt,draft,leqno]{article}
\usepackage{amssymb,eucal,latexsym}
\usepackage{amsmath,amsthm,latexsym,amscd,amsbsy,fleqn,graphicx}
\usepackage{xy}
\xyoption{all}

\newtheorem{theorem}{Theorem}[section]
\newtheorem{lm}[theorem]{Lemma}
\newtheorem{exa}[theorem]{Example}
\newtheorem{exas}[theorem]{Examples}
\newtheorem{cor}[theorem]{Corollary}
\newtheorem{pro}[theorem]{Proposition}
\newtheorem{defi}[theorem]{Definition}

\newtheorem{notas}[theorem]{Notations}
\newtheorem{rem}[theorem]{Remark}
\newtheorem{rems}[theorem]{Remarks}
\newtheorem{fact}[theorem]{Fact}
\newtheorem{facts}[theorem]{Facts}

\newtheorem{constr}[theorem]{Construction}
\newtheorem{confession}[theorem]{Confession}

\newcommand{\df}{\ensuremath{\overset{\mathrm{df}}{=}}}

\def\p{\varphi}
\def\a{\alpha}
\def\b{\beta}

\def\ep{\varepsilon}
\def\g{\gamma}

\def\s{\sigma}

\def\ups{\upsilon}
\def\up{\upsilon}

\def\UP{\Upsilon}

\def\epb{\bar{\varepsilon}_A^{Z_A}}

\def\lra{\longrightarrow}

\def\sbe{\subseteq}

\def\stm{\setminus}
\def\ems{\emptyset}
\def\nes{\neq\emptyset}

\def\cuk{\,\check{}\,}

\def\ovl{\overline}
\def\unl{\underline}

\def\ex{\exists}
\def\fa{\forall}

\def\we{\wedge}

\def\ap{^{\,\prime}}
\def\inv{^{-1}}
\def\st{\ |\ }

\def\nin{\not\in}

\def\A{\mbox{{\boldmath $A$}}}

\def\T{\mbox{{\boldmath $T$}}}

\def\cbf{{\bf c}}

\def\1{{\bf 1}}
\def\2{\mbox{{\sf 2}}}
\def\3{\mbox{{\bf 3}}}

\def\AA{{\cal A}}
\def\BB{{\cal B}}
\def\CC{{\cal C}}
\def\CCC{{\cal C}}
\def\DD{{\cal D}}
\def\EE{{\cal E}}

\def\PP{{\cal P}}

\def\TT{{\cal T}}
\def\UU{{\cal U}}
\def\XX{{\cal X}}
\def\YY{{\cal Y}}

\def\BBB{{\sf B}}
\def\CCC{{\sf C}}

\def\KKK{{\sf K}}

\def\PPP{{\sf P}}

\def\UUU{{\sf U}}

\def\ult{{\mathfrak u}}
\def\ultv{{\mathfrak v}}
\def\ultw{{\mathfrak w}}
\def\clu{{\mathfrak c}}

\def\CO{{\rm CO}}
\def\KO{{\rm KO}}
\def\RC{{\rm RC}}

\def\CR{{\rm CR}}
\def\BClust{{\rm BClust}}
\def\Clust{{\rm Clust}}
\def\BUlt{{\rm BUlt}}

\def\co{{\sf CO}}
\def\Id{{\sf Id}}
\def\rc{{\sf RC}}
\def\rcl{{\sf RCL}}

\def\clust{{\sf Clust}}
\def\bclust{{\sf BClust}}

\def\smf{{\,\smallfrown\,}}
\def\nsmf{{\,\not\smallfrown\,}}

\def\HLC{{\bf LKHaus}}
\def\DHLC{{\bf CLCA}}

\def\CLCA{{\bf CLCA}}

\def\DBoo{{\bf DBoo}}

\def\int{\mbox{{\rm int}}}

\def\cl{\mbox{{\rm cl}}}

\def\Clust{\mbox{{\rm Clust}}}

\def\BClu{\mbox{{\rm BClust}}}
\def\Ult{\mbox{{\rm Ult}}}
\def\Ultsf{\mbox{{\sf Ult}}}

\def\COsf{\mbox{{\sf CO}}}

\def\doc{\hspace{-1cm}{\em Proof.}~~}
\def\sq{\hspace*{\fill} \hbox{\vrule\vbox{\hrule\phantom{o}\hrule}\vrule}}
\def\sqs{\sq \vspace{2mm}}

\def\Top{{\bf Top}}

\def\BBBB{\mathbb{B}}

\def\Set{{\bf Set}}

\def\tcx{t_X^C}
\def\tcy{t_Y^C}
\def\tcx0{t_{(X,X_0)}}
\def\tcy0{t_{(Y,Y_0)}}

\def\di{\diamond}

\def\bU0{\bar{U}=(U^0,(U^i,U^{ci})_{i\in\omega})}

\def\bV0{\bar{V}=(V^0,(V^i,V^{ci})_{i\in\omega})}

\def\CAPX{{\sf C}(\AA,\PP,\XX)}

\def\ZCB{{\bf zCBoo}}
\def\LZCB{{\bf lzCBoo}}

\def\EDT{{\bf EdTych}}
\def\EDL{{\bf EdLKH}}

\title{{\large\bf A Categorical Review of Complete Regularity }\\
\vspace{0.2cm}
\vspace{0.5cm}
{\large Amir Homayoun Nejah and Walter Tholen}\thanks{
The second author acknowledges the support received under the Discovery Grants Program (no. 501260) of the Natural Sciences and Engineering Council of Canada. 
This grant and the assistance of the Fields Institute for the Mathematical Sciences supported the stay of the first author at York University in 2021-22, during which this work was completed.}
\\
\vspace{0.2cm}
{\footnotesize\rm Department of Mathematics and Statistics, York University,}
 {\footnotesize\rm Toronto, Ontario, M3J 1P3, Canada}
\\
\vspace{0.2cm}
}
\date{}

\begin{document}

\maketitle

\begin{abstract}
We use the ultrafilter-convergence axiomatics for topological spaces to motivate in detail a gentle categorical introduction, first to Barr's $\Set$-based relational $T$-algebras, and then to Burroni's $T$-preorders internal to a category $\CC$, here called $T$-spaces in $\CC$, for a monad $T$ on $\CC$ that substitutes the ultrafilter monad on $\Set$. Within these settings one finds not only the notions of compactness and Hausdorff separation, originally due to Manes, but also that of complete regularity. Based on a somewhat hidden result by Burroni, the main theorem of this paper establishes an external fibrational characterization of  the category of completely regular $T$-spaces with its reflexive subcategory of compact Hausdorff $T$-spaces, under modest assumptions on $\CC$ and $T$. 
\end{abstract}



\footnotetext[3]{{\footnotesize {\em E-mail:}
a.homayoun.nejah@gmail.com; tholen@yorku.ca} .}

\section{\large Introduction}\label{intro}
The axiomatization of convergence was a key motivator for the development of the notions of metric and topology, as pursued primarily by Frech\'{e}t and Hausdorff in the early twentieth century. The 1967 thesis by Manes \cite{Manes1967} (see also \cite{Manes1969}) characterized compact Hausdorff spaces as the Eilenberg-Moore algebras with respect to the ultrafilter monad $\UU$ on the category $\Set$ of sets. In this setting, the algebraic structure of a compact Hausdorff space $X$ is given by a single (generalized and infinitary) operation $\UU X\to X$ on its underlying set which assigns to every ultrafilter on $X$ its limit, subject to just two equational axioms that are reminiscent of the unity and associativity axioms for monoids.

Replacing the mapping $\UU X\to X$ by a mere relation between ultrafilters and points, Barr \cite{Barr1970} showed that, with the appropriate relational relaxation of the two equational axioms, one obtains an axiomatization of all topological spaces, with the ultrafilter convergence axioms now generalizing the reflexivity and transitivity conditions of preorders. He thereby reinforced a sentiment that is already apparent in Hausdorff's book \cite{Hausdorff1914}, namely that topological spaces are generalized preordered sets which, in today's parlance, appear as Alexandroff spaces.

Burroni \cite{Burroni1971} took Barr's relational $T$-algebras for a monad $T$ from the category $\Set$ to an arbitrary base category $\CC$ with pullbacks, without any a-priori restrictions on the monad $T$. His $T$-{\em categories} in $\CC$  cover simultaneously categories internal to $\CC$ (when $T$ is the identity monad) and $T$-preorders internal to $\CC$, which give an equivalent description of Barr's relational $T$-algebras in the case $\CC=\Set$. We call these $T$-preorders $T$-{\em spaces} in this paper, in order to emphasize their topological relevance, which reaches far beyond the order-theoretic perspective.

Without any reference to neighbourhood systems, or to open or closed sets, in the largely expository Section 2 we strictly follow the Manes-Barr ultrafilter-convergence axioms to give the notions of compactness, Hausdorff separation and complete regularity, and establish the reflectivity of Hausdorff compactness and complete regularity in this setting. The point of this presentation becomes clear in Section 3 since, in the preceding section, we actually never use any details regarding the ultrafilter monad, other than the general monad conditions. This means that, without any loss in mathematical content, $\UU$ may be replaced by any $\Set$-monad $T$ so that, in effect, topological spaces may be traded for Barr's relational $T$-algebras. Other than the case $T=\UU$, we consider in some detail the case that $T$ is the power-set monad, in order to augment (and, in one instance, rectify) its treatment in \cite{Manes1967}, \cite{Kamnitzer1974}, \cite{Perry1976}.

Minimizing Burroni's categorical apparatus, we describe in a topology-oriented notation the decisive step of replacing $\Set$ by  an abstract category $\CC$ in Section 4. In order to present the category $T$-{\bf Spa}$(\CC)$ of $T$-spaces in $\CC$ as a {\em topological} (\cite{AHS}, \cite{MonTop}) and, hence, ``$\Top$-like'' category over $\CC$, we assume $\CC$ to be complete and well-powered and to have a (regular epi, mono)-factorization system, and impose the preservation of regular epimorphisms by $T$ as the only condition beyond the monad axioms. (This condition comes for free in the case $\CC=\Set$, Choice granted.) Requiring only modest knowledge of basic category theory, we establish the Tychonoff Theorem and the existence of the $\mathrm{\check{C}}$ech-Stone reflection in this setting. Its presence is helpful (but not indispensable) for the introduction of the notion of complete regularity at this level.

Without reference to the case $\CC=\Set$ and $T=\UU$, and under a notation which suggests that he regarded just the identity monad as his role model, Burroni \cite{Burroni1971} characterized the category of completely regular $T$-spaces as a universal (Grothendieck-) fibrational extension of the monadic category of compact Hausdorff $T$-spaces over $\CC$. Our Theorem \ref{maintheorem} gives a detailed analysis, extension and  proof of this result, highlighting in particular its 2-categorical nature. 

We end this paper with some comparative remarks about some related works.

\section{\large A primer on axiomatic ultrafilter convergence}

For a set $X$, we denote by $\UU X$ the set of ultrafilters on $X$. For a map $f:X\to Y$, the map $\UU f:\UU X\to \UU Y$ sends $\mathfrak x\in \UU X$ to its image under $f$, defined by
$$ B\in \UU f(\mathfrak x)\iff f\inv(B)\in\mathfrak x,$$
for all $B\subseteq Y$. The ensuing endofunctor $\UU : \Set\to\Set$ preserves coproducts (disjoint unions) and, by B\"{o}rger's theorem \cite{Borger1987}, $\UU$ is characterized by the remarkable  property of being terminal with respect to the preservation of coproducts: for any coproduct-preserving endofunctor $F$ of $\Set$, there is exactly one natural transformation $F\to \UU $. The $X$-component of the unique transformations $\mathrm{Id}_{\Set}\to \UU $ and $\UU \UU\to \UU$ assigns to $x\in X$ and $\mathfrak X\in \UU \UU X$ respectively the fixed ultrafilter $\dot{x}$ and the filtered sum $\Sigma \mathfrak X$, defined by
$$A\in\dot{x}\iff x\in A\qquad\text{and}\qquad A\in\Sigma \mathfrak X\iff\{\mathfrak x\in \UU X\mid A\in\mathfrak x\}\in\mathfrak X,$$
for all $A\subseteq X$. For later use, we note that, by B\"{o}rger's theorem, the diagrams
\begin{center}
$\xymatrix{\UU\ar[rr]^{\UU\dot{(-)}}\ar[d]_{\dot{(-)}\UU}\ar[rrd]^{1_{\UU}} && \UU \UU\ar[d]^{\Sigma} &&& \UU\UU\UU\ar[rr]^{\UU\Sigma}\ar[d]_{\Sigma \UU} && \UU\UU\ar[d]^{\Sigma}\\
\UU\UU\ar[rr]_{\Sigma} && \UU &&& \UU\UU\ar[rr]_{\Sigma} && \UU\\
}$	
\end{center}
commute and, therefore, provide $\UU$ uniquely with the structure of a {\em monad} (we recall the definition of monad more generally in Section 3).

\begin{defi}\label{definitionUspace}
\rm 
	(1) A set $X$ together with a relation $C\subseteq \UU X\times X$ is called a $\UU${\em -graph}; thinking of $C$ as a convergence relation, we often write $\mathfrak x\rightsquigarrow y $ for $(\mathfrak x,y)\in C$. The $\UU$-graph $(X,\rightsquigarrow)$ is
	\begin{itemize}
	\item[(R)] {\em reflexive} if $\dot{x}\rightsquigarrow x$, for all $x\in X;$	
	\item[(T)] {\em transitive} if $(\mathfrak X\,\hat{\rightsquigarrow}\, \mathfrak y\;\text{and}\; \mathfrak y\rightsquigarrow z\Longrightarrow \Sigma \mathfrak X\rightsquigarrow z)$, for all $\mathfrak X\in \UU\UU X,\mathfrak y\in \UU X, z\in X$;  
	\end{itemize}
	here we have extended the relation $C$ to a relation $\hat{C}\subseteq \UU\UU X\times \UU X$ and written $\mathfrak X\,\hat{\rightsquigarrow}\,\mathfrak y$ instead of $(\mathfrak X,\mathfrak y)\in \hat{C}$, as an abbreviation of the statement
	$$\exists\; \mathfrak w\in \UU C: \UU\pi_1(\mathfrak w)=\mathfrak X\;\;\text{and}\;\; \UU\pi_2(\mathfrak w)=\mathfrak y,$$
	where $\pi_1:C\to \UU X,\;\pi_2:C\to X$ denote the restricted product projections. If there is no danger of ambiguity, for simplicity we may write $\mathfrak X\rightsquigarrow y$, rather than $\mathfrak X\,\hat{\rightsquigarrow}\,\mathfrak y$.
	
	(2) A map $f:(X,\rightsquigarrow)\to(Y,\rightsquigarrow)$ of $\UU$-graphs is {\em monotone} if it preserves the convergence relation:
	$$ \mathfrak x\rightsquigarrow y\Longrightarrow \UU f(\mathfrak x)\rightsquigarrow fy,$$
	for all $\mathfrak x\in \UU X, y\in X$. Taking as morphisms the monotone maps, one obtains the category
	\begin{itemize}
	\item $\UU$-{\bf Gph} of all $\UU$-graphs, as well as its full subcategories	
	\item $\UU$-{\bf RGph} of the reflexive $\UU$-graphs, and
	\item $\UU$-{\bf Spa} of the reflexive and transitive $\UU$-graphs which, more briefly, we refer to as
	$\UU$-{\em spaces}. 
	\end{itemize}
\end{defi}
We state some basic properties of these categories and their underlying-set functors.
\begin{facts}\label{USpatopological}
\rm 
	(1) Every set $X$ that comes with a family of maps $f_i:X\to Y_i \,(i\in I)$, where each $Y_i$ is actually a $\UU$-graph, and where $I$ is allowed to be a proper class or empty, may be provided with the {\em initial} (also {\em weak} or, in the case of a singleton family, {\em cartesian}) $\UU$-graph structure with respect to $(f_i)_{i\in I}$ (see Definition \ref{cartesiandefi} and Remark \ref{topologicalfunctorrem} for the general definitions), described by
	$$\mathfrak x\rightsquigarrow y\iff \forall\,i\in I: \UU f_i(\mathfrak x)\rightsquigarrow f_iy,$$
	for all $\mathfrak x\in \UU X, y\in X$. This $\UU$-graph structure on $X$ is largest with the property of making all maps $f_i$ monotone, and it is reflexive (and transitive) if the $\UU$-graph structure of each $Y_i$ is. (To check the transitivity claim, denote by $\tilde{f_i}:C\to C_i$ the restriction of the map $\UU f_i\times f_i$ to the convergence relations and observe that, when $\UU\pi_1(\mathfrak w)\rightsquigarrow \UU\pi_2(\mathfrak w)$ in $X$, then $\UU\pi_1^i(\UU\tilde{f_i}(\mathfrak w))\rightsquigarrow \UU\pi_2^i(\UU\tilde{f_i}(\mathfrak w)) $ in $Y_i$, for every $\mathfrak w\in \UU C$ and $i\in I$.) This means that each of the categories $\UU$-{\bf Gph}, $\UU$-{\bf RGph} and $\UU$-{\bf Spa} is {\em topological} over $\Set$ (for this notion, see Remark \ref{topologicalfunctorrem} or, for example,  \cite{AHS} and \cite{MonTop}). 
	
	(2) As topologicity is a self-dual property, we note that for a set $Y$ that comes with a family of maps $f_i:X_i\to Y\,(i\in I)$, where each $X_i$ is a (reflexive; reflexive and transitive) $\UU$-graph, there is a least structure on $Y$ (with the corresponding property) making all maps $f_i$ monotone, called {\em final} or, in the case of a singleton family, {\em cocartesian}. But this structure does not have as nice a description as its dual counterpart, especially for $\UU$-spaces. Still, topologicity over $\Set$ makes each of the categories $\UU$-{\bf Gph}, $\UU$-{\bf RGph}, and $\UU$-{\bf Spa}
	complete and cocomplete, with limits (colimits) formed as in $\Set$ and then provided with the initial (final) structure with respect to the limit projections (colimit injections, respectively).
	
	(3) The category $\UU$-{\bf RGph} is not just topological over $\Set$ but, unlike $\UU$-{\bf Gph} and $\UU$-{\bf Spa}, also a {\em quasitopos} (see, for example, \cite{AHS} or \cite{MonTop}). In particular, $\UU$-{\bf RGph} is cartesian closed and, hence, allows for the formation of ``well-behaved'' function spaces $\UU\text{-{\bf Spa}}(X,Y)$	for all of its objects $X,Y$, with the function-space structure given by the initial structure with respect to the family of evaluation maps. 
	\end{facts}

\begin{defi}\label{definitionUAlg} 
\rm 
We call a $\UU$-space $(X,\rightsquigarrow)$ {\em algebraic} if its convergence relation is (the graph of) a map; that is, if 
\begin{itemize}
\item[(A)] there is a map $c:\UU X\to X$ with $(\mathfrak x\rightsquigarrow y \iff c(\mathfrak x)= y),$
for all $\mathfrak x\in \UU X, y\in X$;
\end{itemize} equivalently (in the presence of Choice), if every ultrafilter has at least, and at most, one point of convergence:
\begin{itemize}
\item[(K)] for every $\mathfrak x\in \UU X$, there is a (tacitly chosen) point  $y\in X$ with $ \mathfrak x\rightsquigarrow y;$
\item[(H)] if an ultrafilter $\mathfrak x\in \UU X$ satisfies $ \mathfrak x\rightsquigarrow y$ and $\mathfrak x\rightsquigarrow z$, then necessarily $y=z\,$.
\end{itemize}
Clearly, such a map $c$ (if it exists) is uniquely determined by the convergence relation. 

\medskip
We have the following full subcategories of $\UU$-{\bf Spa}:
\begin{itemize}
\item $\UU$-{\bf KSpa}, containing all the $U$-spaces satisfying condition (K);
\item $\UU$-{\bf HSpa}, containing all the $U$-spaces satisfying condition (H);
\item $\UU$-{\bf ASpa}, containing all algebraic $U$-spaces; that is:
\begin{center}
	$\UU$-{\bf ASpa} $=\UU$-{\bf KSpa}\;$\cap$	$\UU$-{\bf HSpa}.
		\end{center}
 
\end{itemize}
\end{defi}
At the moment, we are primarily interested in the category $\UU$-{\bf Alg} and first give an explanation for the name ``algebraic''.

\begin{pro}\label{Ualgebras}
 A map $c: \UU X\to X$ makes a set $X$ an algebraic $\UU$-space via $\mathrm{(A)}$ if, and only if, $c$ makes the diagram 
\begin{center}
	$\xymatrix{X\ar[r]^{\dot{(-)}}\ar[rd]_{1_X} & \UU X\ar[d]^{c} && \UU\UU X\ar[ll]_{Uc}\ar[d]^{\Sigma} &&& \UU X\ar[rr]^{\UU f}\ar[d]_{c} && \UU Y\ar[d]^{d}\\
	& X && \UU X\ar[ll]_{c} &&& X\ar[rr]^f && Y\\
	}$
\end{center}
on the left commutative. Furthermore, given also a $\UU$-space $Y$ that is algebraic qua $d:\UU Y\to Y$, then a map $f:X\to Y$ is monotone if, and only if, the diagram on the right commutes.
\end{pro}

\begin{proof}
To say that $C\subseteq \UU X\times X$ is the graph of a map $c:\UU X\to X$ means that $C$ is the image  of the map $\langle 1_{\UU X},c\rangle:	 \UU X\to \UU X\times X$, which lets us identify $C$ with $\UU X$, and the projections $\pi_1:C\to \UU X,\,\pi_2:C\to X$ with the maps $1_{\UU X}, c$, respectively. But this makes immediately clear that then the relation $\hat{C}\subseteq \UU\UU X\times \UU X$ of Definition \ref{definitionUspace} is the graph of the map $\UU c:\UU\UU X\to \UU X$. Consequently, conditions (R) and (T) translate respectively  to $c(\dot{x})=x$ and $c(\UU c(\mathfrak X))=c(\Sigma \mathfrak X)$ for all $x\in X$ and $\mathfrak X\in \UU\UU X$, which is the commutativity of the left diagram. The monotonicity claim is obvious.
\end{proof}
\begin{rems}\label{laxUalgebras}
\rm 
(1) Pairs $(X,c)$, consisting of a set $X$ and a map $c:\UU X\to X$ that lets the above diagram on the left commute, are the objects of the {\em Eilenberg-Moore category}\; $\Set^{\UU}$ of the monad $\UU$, usually referred to as $\UU$-{\em algebras}; their (homo)morphisms are described by the diagram on the right (see \cite{MacLane1971}) . Hence, Proposition \ref{Ualgebras} implies that the categories
$$\Set^{\UU}\cong \UU\text{-{\bf ASpa}}$$
are concretely isomorphic, so that the induced-graph functor functor $(X,c)\longmapsto(X,C)$ is bijective and commutes with the underlying-set functors.

(2) If one writes a relation $R\subseteq X\times Y$ as an arrow $R:X\to Y$ (that is: as a morphism in the category {\bf Rel} of sets and relations, with their relational composition) and identifies maps of sets with their graphs, then the objects and morphisms of the category $\UU$-{\bf Spa} are equivalently described by the {\em laxly} commuting diagrams
\begin{center}
	$\xymatrix{X\ar[r]^{\dot{(-)}}\ar[rd]_{1_X}^{\subseteq} & \UU X\ar[d]^{C\qquad\;\,\subseteq} && \UU\UU X\ar[ll]_{\hat{\UU} C}\ar[d]^{\Sigma} &&& \UU X\ar[rr]^{\UU f}\ar[d]_{C_X}^{\qquad\;\;\;\subseteq} && \UU Y\ar[d]^{C_Y}\\
	& X && \UU X\ar[ll]^{C} &&& X\ar[rr]_f && Y\\
	}$
\end{center}
so that $1_X\subseteq C\dot{(-)},\;C(\hat{\UU}C)\subseteq C\Sigma$, and $fC_X\subseteq C_Y(\UU f)$; here we have written $\hat{\UU}C$ for the relation $\hat{C}$ of Definition \ref{definitionUspace}, to emphasize the extension of $\UU$ from maps to relations. Briefly: the category $\UU$-{\bf Spa} is  presented as the category of {\em lax relational $\UU$-algebras}.
\end{rems}

In what follows, we investigate further the full embeddings
\begin{center}
$\xymatrix{ && \UU\text{-{\bf KSpa}}\ar[rrd] && && \\
\UU\text{-\bf{ASpa}}\ar[rru]\ar[rrd] && && \UU\text{-{\bf Spa}}\ar[rr] && \UU\text{-{\bf RGph}} \\
&& \UU\text{-{\bf HSpa}}\ar[rru] && &&\\
  }$	
\end{center}
with the primary goal of describing a reflector (= left adjoint) $\UU\text{-{\bf RGph}}\to \UU$-{\bf Alg} as concretely as possible in this context. We start by picking some low-hanging fruit:

\begin{pro}\label{lowhangingfruit}
\begin{itemize}
\item[{\em (1)}] For a surjective monotone map $f:X\to Y$ of $\UU$-spaces, if $X$ satisfies {\em (K)}, so does $Y$.
\item[{\em(2)}] The direct product of a family of $\UU$-spaces satisfying {\em (K)} satisfies {\em (K)} again.
\item[{\em (3)}] For a (possibly large or empty) point-separating family of monotone maps $f_i:X\to Y_i\;(i\in I)$ of $\UU$-spaces, if every $Y_i$ satisfies {\em (H)}, so does $X$.
\end{itemize}	
\end{pro}

\begin{proof} (1) Since (with Choice granted) the surjection $f$ has a section in $\Set$, which the functor $\UU$ must preserve, also the map $\UU f$ is surjective. Hence, every $\mathfrak y\in \UU Y$ may be written as $\mathfrak y=\UU f(\mathfrak x)$, with $\mathfrak x\rightsquigarrow x$ for some $x\in X$, by hypothesis; now $\mathfrak y\rightsquigarrow fx$ follows. Using Facts \ref{USpatopological} (1), (2) (and Choice) one proves the assertions
(2) and (3) similarly. Detailed proofs	appear under Theorem \ref{lowhangingfruitgeneralized} in a more general context.
\end{proof}
\begin{cor}\label{Hausdorffreflection}
$\UU\text{-{\bf HSpa}}$ is (regular epi)-reflective in $\UU\text{-{\bf Spa}}$, and $\UU\text{-{\bf ASpa}}$ is (regular epi)-reflective in $\UU\text{-{\bf KSpa}}$.
\end{cor}
\begin{proof}
As is well known, the two assertions follow successively from items (3) and (1) of the Proposition, by Adjoint-Functor-Theorem methods (see, for example, \cite{MacLane1971}, p. 135). We nevertheless sketch a short proof, as follows. Define an equivalence relation on a given $\UU$-space	 $X$ by
$$ x\sim x'\iff\forall f:X\to Y \text{ monotone, } Y\in \UU\text{-{\bf HSpa}}: fx=fx'$$
and provide $X/\!\sim$ with the final $\UU$-space structure with respect to the projection $p:X\to X/\!\sim$, which makes $p$ a regular epimorphism in $\UU$-{\bf Spa}. By design, every morphism $f:X\to Y$ with $Y\in \UU\text{-{\bf HSpa}}$ factors uniquely through $p$, by a morphism $f^{\ast}:X/\!\sim\,\longrightarrow Y$. Moreover, $X/\!\sim$ satisfies (H); indeed, if $\UU p(\mathfrak x)\rightsquigarrow px$ and $\UU p(\mathfrak x)\rightsquigarrow px'$ in $X/\!\sim$, then for every $f:X\to Y$ as above,
$$\UU f(\mathfrak x)=\UU f^{\ast}(\UU p(\mathfrak x))\rightsquigarrow f^{\ast}(px)=fx\quad\text{and}\quad \UU f(\mathfrak x)=\UU f^{\ast}(\UU p(\mathfrak x))\rightsquigarrow f^{\ast}(px')=fx',$$
which implies $fx=fx'$ since $Y$ satisfies (H); hence, $x\sim x'$ follows. Consequently. $p$ serves as a reflection of $X$ into $\UU\text{-{\bf HSpa}}$. Furthermore, if $X$ happens to satisfy (K), so does $X/\!\sim$, by the Proposition.
 \end{proof}
 \begin{rems}
 \rm 
(1) The definition of the equivalence relation $\sim$ in the above proof formally uses a universal quantification over a proper class of data. But the presentation $\sim\; =\bigcap_f\!\sim_f$ (with $\sim_f$ denoting the equivalence relation on $X$ induced by $f$) shows that (with Choice) one may restrict oneself to using a set of data.

(2) Their reflectivity makes the subcategories of Corollary \ref{Hausdorffreflection} closed under limits in their ambient categories. While property (K) is stable under direct products in  $\UU\text{-{\bf Spa}}$, the equalizer of parallel morphisms in $\UU\text{-{\bf KSpa}}$ formed in $\UU\text{-{\bf Spa}}$ generally fails to satisfy (K) (as will become obvious in Section 3). 
 \end{rems}
In order to present the left adjoint of $\UU\text{-{\bf ASpa}}\hookrightarrow \UU\text{-{\bf RGph}}$ we will work with the isomorphic copy $\Set^{\UU}$ of $\UU\text{-{\bf Alg}}$ and first note that, for every set $X$, one has the {\em free} $\UU$-algebra $(\UU X,\Sigma_X)$ over $X$, so that every map $f:X\to Y$ to a $\UU$-algebra $(Y,c)$ factors as $f=f^{\sharp}\dot{(-)}$, with a uniquely determined homomorphism $f^{\sharp}:\UU X\to Y$; namely, $f^{\sharp}=(\UU f)c$. This follows formally from the monad structure of $\UU$.

Let us now assume that $X$ carries the structure $C=\,\rightsquigarrow$ of a reflexive $\UU$-graph and form the equivalence relation (= kernel pair in $\Set$) $K=\,\sim_{\pi_2}$ induced by $\pi_2:C\to X$; that is: $K=\{((\mathfrak x,y),(\mathfrak z,y))\mid \mathfrak x\rightsquigarrow y\text{ and }\mathfrak z\rightsquigarrow y \text{ in } X\}$. With the projections $p_1, p_2: K\to C$ and $\pi_1: C\to UX$ we can then state:

\begin{lm}\label{StoneCechLemma}
 For every homomorphism $q:(UX,\Sigma_X)\to (Q,c)$ onto a $\UU$-algebra $(Q,c)$, the following assertions are equivalent:
\begin{itemize}
\item[{\em (i)}] $\sim_q$ is the least congruence relation on $(\UU X,\Sigma_X)$ making the map $q\dot{(-)}:X\to Q$ monotone;
\item[{\em (ii)}] $q$ is the coequalizer of $(\pi_1p_1)^{\sharp},(\pi_1p_2)^{\sharp}:(\UU K,\Sigma_K)\rightrightarrows (\UU X,\Sigma_X)$ in $\Set^{\UU}$, where $f^{\sharp}$ denotes the mate of $f$ under adjunction (see their defining property in the diagram below);
\item[{\em (iii)}] $q\dot{(-)}:X\to Q$ is the reflection of morphism f $(X,\rightsquigarrow)$ into $\UU\text{-{\bf ASpa}}\cong\Set^{\UU}$.
\end{itemize}	
\end{lm}
\begin{proof}
For any map $f:X\to Y$ into a $\UU$-algebra $(Y,c_Y)$, consider the diagram
\begin{center}
$\xymatrix{	\UU K\ar@/^0.3pc/[rr]^{(\pi_1p_1)^{\sharp}}\ar@/_0.3pc/[rr]_{(\pi_1p_2)^{\sharp}\quad} && \UU X\ar[rr]^{f^{\sharp}} && Y\\
& K\ar[ul]^{\dot{(-)}}\ar@/^0.3pc/[ru]|{\pi_1p_1}\ar@/_0.3pc/[ru]_{\pi_1p_2} && X\ar[ul]^{\dot{(-)}}\ar[ur]_f & \\
}$
\end{center}
By definition, monotonicity of $f$ means that $c_Y(\UU f(\mathfrak x)) = fx$ holds whenever $\mathfrak x\rightsquigarrow x$ in $X$; equivalently, that $f^{\sharp}(\mathfrak x)= f^{\sharp}(\dot{x})$ whenever $\mathfrak x\rightsquigarrow x$, or that the congruence relation $\sim_{f^{\sharp}}$ on $(\UU X,\Sigma_X)$ induced by the homomorphism $f^{\sharp}$ contains all pairs $(\mathfrak x, \dot{x})$ with $\mathfrak x\rightsquigarrow x$. A further equivalent condition for the monotonicity of $f$ is the equality 
$$ f^{\sharp}(\pi_1p_1)^{\sharp}=f^{\sharp}(\pi_1p_2)^{\sharp}\quad\text{or, equivalently,}\quad f^{\sharp}\pi_1p_1=f^{\sharp}\pi_1p_2;$$
indeed, the latter equality means precisely $f^{\sharp}(\mathfrak x)=f^{\sharp}(\mathfrak z)$ whenever $\mathfrak x\rightsquigarrow x$ and $\mathfrak z\rightsquigarrow x$ in $X$, and one may equivalently specialize $\mathfrak z$ to $\dot{x}$.	

Now, considering in particular the map $e=q\dot{(-)}$, so that $e^{\sharp}=q$, as well as the fact that, for a monotone map $f:X\to Y$, the homomorphism $f^{\sharp}$ factors through $q$ by a (necessarily unique) homomorphism $\overline{f}:Q\to Y$ precisely when $\sim_q\,\subseteq\,\sim_{f^{\sharp}}$, one routinely confirms the claimed equivalence of statements (i--iii).
\end{proof}

The Lemma suggests two methods of proof of the following theorem, (i) one that applies Adjoint-Functor-Theorem methods , and (ii) another one that relies on the existence of coequalizers in $\UU\text{-{\bf ASpa}}$, which may be secured in, again, two different ways again. We sketch all three resulting methods, but note that none of them may be considered constructive.

\begin{theorem}\label{StoneCechTheorem}
	The category $\UU\text{-{\bf ASpa}}$ is reflective in $\UU\text{-{\bf RGph}}$ and, hence, also in  $\UU\text{-{\bf Spa}}$.
\end{theorem}

\begin{proof} 
The reflectivity of $\UU\text{-{\bf ASpa}}$ in $\UU$-{\bf Spa} follows trivially from its reflectivity in $\UU$-{\bf RGph}. It therefore suffices to consider a reflexive $\UU$-graph $(X,\rightsquigarrow)$ and construct its reflection into $\UU$-{\bf ASpa}, as indicated alternatively in (i) or (ii) below. 

(i) Consider a representative set of surjective homomorphisms $q_i: (\UU X,\Sigma)\to (Q_i,c_i)\;(i\in I)$ such that $q_i\dot{(-)}:X\to Q_i$ is monotone. Then let $(Q,c)$ be the image of the induced homomorphism $\UU X\to\prod_{i\i I}Q_i$ and $q:\UU X\to Q$ its restriction.

(ii) Form the coequalizer $q$ as in Lemma \ref{StoneCechLemma} (ii) --- which, however, raises the question of how to form the coequalizer of a pair of homomorphisms $g,h: Z\rightrightarrows W$ in $\UU\text{-{\bf ASpa}}$. This may again be done by Adjoint-Functor-Theorem methods: one considers a representative set of surjective homomorphisms $r_i: W\to R_i\;(i\in I)$ with $r_ig=r_ih$ and proceeds as in (i). Alternatively, one may form the coequalizer $s:W\to S$ in $\Set$ and consider $S$ as a $\UU$-space with the final structure with respect to $s$. By Proposition \ref{lowhangingfruit}, $S$ satisfies (K) and just needs to be subjected to the reflection $r:S\to R$ into $\UU\text{-{\bf ASpa}}$ of Corollary \ref{Hausdorffreflection}, to obtain the coequalizer $rs:W\to R$ of $g,h$ in $\UU\text{-{\bf ASpa}}$.
\end{proof}

\begin{defi}\label{Ucomplreg}
\rm
	Denoting its reflection into $\UU\text-{\bf ASpa}$ by $\beta_X:X\to \mathrm B X$ (with $\mathrm B$ to be read as capital Beta\footnote{We write $\mathrm BX$, rather than the standard $\beta X$, in order to be able to distinguish notationally the functor $\mathrm B$ from the natural transformation $\beta$; not doing so may cause confusion in Section 5. }), we say that a $\UU$-space $X$ satisfies condition
	\begin{itemize}
	\item[(C)] if $X$ carries the cartesian (=\; initial =\;weak) structure with respect to $\beta_X$; that is, if for all $\mathfrak x\in \UU X, y\in X$, one has	$\mathfrak x\rightsquigarrow y$ in $X$ whenever $\UU\beta_X(\mathfrak x)\rightsquigarrow \beta_Xy$ in $\mathrm BX$;
	\item[(F)] if $\beta_X$ is a monomorphism; that is: if $\beta_X$ an injective map;	
	\item[(CF)] if $X$ satisfies (C) and (F).	
	\end{itemize}
We denote the corresponding full subcategories of $\UU\text{-{\bf Spa}}$ by
\begin{center}
$\UU\text{-{\bf CSpa}}$, $\UU\text{-{\bf FSpa}}$, and $\UU\text{-{\bf CFSpa}}=\UU\text{-{\bf CSpa}}\cap \UU\text{-{\bf FSpa}}.$
\end{center}
\end{defi}

Trivially, every algebraic $\UU$-space $X$ satisfies (CF), since being algebraic means equivalently that $\beta_X$ is an isomorphism. Hence, we have the diagram 
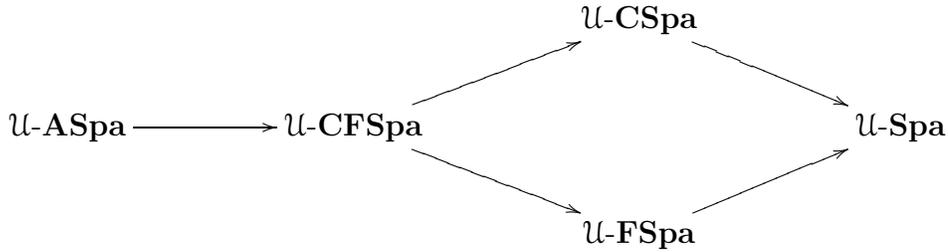
\begin{center}
$\xymatrix{ && && \UU\text{-{\bf CSpa}}\ar[rrd] && \\
\UU\text{-\bf{ASpa}}\ar[rr] && \UU\text{-{\bf CFSpa}}\ar[rru]\ar[rrd] && && \UU\text{-{\bf Spa}} \\
&& && \UU\text{-{\bf FSpa}}\ar[rru] && \\
  }$	
\end{center}
of full embeddings. 
Property (F) implies (H) of Definition \ref{definitionUAlg} (by Proposition \ref{lowhangingfruit} (3)), but generally not conversely; also, property (C) is generally incomparable with property (K) (as will become clear in Section 3). Property (F) enjoys the same stability behaviour as (H) that was stated in Proposition \ref{lowhangingfruit}(3); property (F) is therefore stable under direct products and subobjects. This latter stability holds also for property (C); but, unlike (F) or (K), property (C) satisfies in fact a stability property stronger than product- or subobject stability, as follows:
\begin{theorem}\label{lowhangingfruitbis}
Consider (a possibly large or empty) family of monotone maps $f_i:X\to Y_i\;(i\in I)$ of $\UU$-spaces with the given common domain $X$. 
\begin{itemize}
\item[{\em(1)}] Let $X$ carry the weak (= initial) structure with respect to $(f_i)_{i\in I}$. Then, if every $Y_i$ satisfies {\em (C)}, so does $X$. 
\item[{\em(2)}] Let $(f_i)_{i\in I}$ be point-separating. Then, if every $Y_i$ satisfies {\em (F)}, so does $X$.
\item[{\em (3)}]
$\UU\text{-{\bf CSpa}}$ is simultaneously epi- and mono-reflective in $\UU\text{-{\bf Spa}}$, and $\UU\text{-{\bf FSpa}}$ is regular-epi-reflective in $\UU\text{-{\bf Spa}}$.	 
\end{itemize}

\end{theorem}
\begin{proof}
As the reflectivity claims in (3) follow easily (by Adjoint-Functor-Theorem methods) from the stability properties (1) and (2) (as we show more generally in Theorem \ref{complregrefl}) , we just indicate the proof of these. They, in turn, are an immediate consequence of the functoriality of $\mathrm B: \UU\text{-{\bf Spa}}\to \UU\text{-{\bf Alg}}$ and the naturality of $\beta$. Indeed, for every map $f_i$, consider the $\beta$-naturality diagrams
\begin{center}
$\xymatrix{X\ar[rr]^{f_i}\ar[d]_{\beta_X} && Y_i\ar[d]^{\beta_{Y_i}}\\
\mathrm BX\ar[rr]^{\mathrm B f_i} && \mathrm BY_i\;.\\
}$	
\end{center}
 Now, when $(f_i)_{i\in I}$ is collectively monic and every map $\beta_{Y_I}$ is monic, the composite family $\beta_{Y_i}f_i=(\mathrm Bf_i)\beta_X\; (i\in I)$ is still collectively monic  and, hence, makes $\beta_X$ monic. This shows the stability property for (F), as stated in (2). But the same argumentation may be used for (C), by trading the word ``(collectively) monic" for ``initial'' everywhere; this proves (1).
\end{proof}

\begin{rem}
\rm 
Regular epimorphisms in $\UU\text{-{\bf Spa}}$, or in any topological category over $\Set$, are precisely the quotient maps, {\em i.e.}, the surjective final maps. Simultaneously epic and monic morphisms in such categories are precisely the bijective morphisms. 
\end{rem}

\begin{cor}\label{lowcoro}
$\UU\text{-{\bf CFSpa}}$ is epi-reflective in $\UU\text{-{\bf Spa}}$.
\end{cor}

\begin{proof}
By the Theorem, $\UU\text{-{\bf CFSpa}}$ is closed in $\UU\text{-{\bf Spa}}$ under point-separating initial families and, hence, epi-reflective, by the Adjoint Functor Theorem. Alternatively, with the reflections mentioned in Theorem \ref{lowhangingfruitbis} already at hand, first reflect a given $\UU$-space into $\UU\text{-{\bf FSpa}}$ and then the resulting space into $\UU\text{-{\bf CSpa}}$, to obtain a space that actually lies in $\UU\text{-{\bf CFSpa}}$. The composite of the two reflection morphisms is still surjective.
\end{proof}
We denote the $\UU\text{-{\bf CFSpa}}$-reflection of a $\UU$-space $X$ by $\alpha_X:X\to \mathrm A X$ (with $\mathrm A $ to be read as capital Alpha). 

Our primer on axiomatic ultrafilter convergence ends here, as it suffices for the purposes of this paper. Further topological properties in a convergence setting were considered in the thesis \cite{Kamnitzer1974} and, more comprehensively, in the thesis \cite{Mobus1981}.

\section{Confessions}
With apologies to the patient reader, we should make some belated confessions, the first of which is badly overdue: 
\begin{confession}\label{confession1}
Essentially all assertions of Section {\em 2} are well known, since 
\begin{itemize}
\item $\UU\text{-{\bf RGph}}\cong\mathbf{PsTop}$ is Choquet's {\em \cite {Choquet1948}} category of pseudotopological spaces and their continuous maps (identified by Wyler {\em \cite{Wyler1976}} as the quasitopos hull of the category {\bf Top} of all topological spaces);
	\item $\UU\text{-{\bf ASpa}}\cong\Set^{\UU}$ is (isomorphic to) the category $\mathbf{KHaus}$ of compact Hausdorff spaces and their continuous maps (as first established by Manes {\em \cite{Manes1967}, \cite{Manes1969}});
	\item $\UU\text{-{\bf Spa}}$ is (isomorphic to) {\bf Top} (as first established by Barr {\em \cite{Barr1970}}); consequently,
	\item $\UU\text{-{\bf KSpa}}$	and $\UU\text{-{\bf HSpa}}$ are respectively the full subcategories {\bf KTop} and {\bf Haus} of compact and of Hausdorff topological spaces; and	
	\item $\UU\text{-{\bf CSpa}}$	, $\UU\text{-{\bf FSpa}}$, and $\UU\text{-{\bf CFSpa}}$ are respectively the full subcategories {\bf CReg}, {\bf FHaus}, and {\bf Tych} of completely regular, of functionally Hausdorff (see below), and of Tychonoff spaces.	
	\item The reflections $\beta$ and $\alpha$ take on their expected topological meaning of $\check{C}$\!ech-Stone reflection and of Tychonoff reflection for all (pseudo)topological spaces (as first described in categorical generality by Herrlich {\em \cite{Herrlich1968}}).
	\end{itemize}	
\end{confession}
 
As a reminder: a topological space $X$ is {\em functionally Hausdorff}	 if, for any pair of distinct points $x,y$, there is a continuous map $f:X\to [0,1]$ with $fx=0$ and $fy=1$. We note that even a regular Hausdorff space may fail to be functionally Hausdorff: see, for example, \cite{Willard1968}, Exercises 14G, 18G.
\medskip

The careful reader of Section 2 will have also  observed that, once we had established (right at the beginning of the section) that the ultrafilter functor $\UU$ carries a monad structure, we actually never used  the definition of $\UU$ again; we just used the existence of its monad structure. Therefore:
\begin{confession}\label{confession2}
	Every statement of Section {\em 2} remains true if one trades the monad $\UU$ for any monad $T$ of the category $\Set$ (and replaces the word ``ultrafilter'' by a neutral term, like ``element''); that is: if $\UU$ is replaced by any endofunctor $T$ of $\Set$ that comes equipped with natural transformations $\eta:\mathrm{Id}_{\Set}\to T$ and $\mu:TT\to T$ (replacing $\dot{\text{\em (-)}}$ and $\Sigma$, respectively), satisfying the commutativity conditions
	\begin{center}
$\xymatrix{T\ar[rr]^{T\eta}\ar[d]_{\eta T}\ar[rrd]^{1_T} && TT\ar[d]^{\mu} &&& TTT\ar[rr]^{T\mu}\ar[d]_{\mu T} && UU\ar[d]^{\mu}\\
TT\ar[rr]_{\mu} && T &&& TT\ar[rr]_{\mu} && T\\
}$	
\end{center}
This fact will be no surprise to readers familiar with Barr's relational algebras of {\em \cite{Barr1970}}, or with the more general $(T,V)$-categories (or -spaces) studied in {\em \cite{CT2003},\cite{CHT2004},\cite{MonTop}}.
\end{confession}
In what follows, we use the $T$-for-$\UU$ exchange also in our terminology and notation, so that in particular we have the category $T$-{\bf Spa} of $T$-{\em spaces} with their monotone maps, and{ its full subcategories  $T$-{\bf XSpa}, for {\bf X} = {\bf A}, {\bf K}, {\bf H}, {\bf C}, {\bf F}, or {\bf CF}.

\begin{exas}\label{examples}
\rm 
(1) There are only two (isomorphism types of) $\Set$-monads $T$ for which the unit transformation  $\mathrm{Id}_{\Set}\to T$ fails to be monic; these are the endofunctors $\TT_0, \TT_1$ with $\TT_0X=\TT_1X=1$ (= a fixed singleton set) for every non-empty set $X$, and with $\TT_0\emptyset=\emptyset,\, \TT_1\emptyset=1$. For both of these monads $T$, $T\text{-{\bf Spa}}=T\text{-{\bf RGph}}$ is (isomorphic) to $\Set$. For $T=\TT_0$, every set $X$ satisfies (K) and (C), but each of (H), (F), (CF), (A) is synonymous with $|X|\leq 1$. For $T=\TT_1$, every set $X$ satisfies (C), while (K) means $X\neq\emptyset$; satisfaction of any of (H), (F), (T) means $|X|\leq 1$ again, so that (A) is satisfied precisely when $|X|=1$.

\medskip

(2)	For $T={\rm Id}_{\Set}$, considered as a monad with identical transformations, $T\text{-{\bf Spa}}$ is the category $\mathbf{Ord}$ of preordered sets and their monotone maps. Every preordered set satisfies condition (K), but it will satisfy any of (H), (F), (CF), (A) only if it is discrete. A preorder satisfies (C) precisely when it is an equivalence relation. The reflector $\beta=\alpha$ assigns to a preordered set $X$ the (discretely ordered) set $\pi_0X$ of its connected components. $T\text{-{\bf RGph}}$ is the quasitopos of sets equipped with a reflexive relation, called {\bf Rere} in \cite{AHS}.

\medskip

(3) Consider for $T$ the power-set functor $\PP:\Set\to\Set$, with its monad structure given by
$X\to\PP X,\, x\mapsto\{x\}$, and $\bigcup:\PP\PP X\to\PP X$, for every set $X$. The defining axioms (R) and (T) of a $\PP$-space $(X,\rightsquigarrow)$ may be stated as
$$\text{(R)}\;\{x\}\rightsquigarrow x\quad\text{and}\quad\text{(T)}\; (A_i\rightsquigarrow y_i\;(i\in I)\;\text{and}\;\{y_i\mid i\in I\}\rightsquigarrow z)\Longrightarrow \bigcup_{i\in I}A_i\rightsquigarrow z,$$
for all $x, z, y_i\in X, A_i\subseteq X\,(i\in I)$; a monotone map $f:X\to Y$ satisfies $(A\rightsquigarrow y\Longrightarrow f(A)\rightsquigarrow fy)$. Reading $A\rightsquigarrow y$ as ``$A$ supports $y$'' (and perhaps even writing $A\leq y$ instead of $A\rightsquigarrow y$), one may think of a $\PP$-space as a {\em pluri-ordered set} and denote the resulting category by {\bf PluOrd}. In fact, there are multiple ways of regarding preordered sets as pluri-ordered. Here we concentrate on the well-known fact
that the category $\Set^{\PP}$ of $\PP$-algebras is (isomorphic to) the category {\bf Sup} of complete lattices and their suprema-preserving maps. One has the chain of full embeddings
$$\mathbf{Sup}\longrightarrow\mathbf{Pos}_{\mathrm{sup}}\longrightarrow\mathbf{Ord}_{\mathrm{sup}}\longrightarrow \mathbf{PluOrd},$$
where the objects of $\mathbf{Ord}_{\mathrm{sup}}$ are preordered sets, and their morphisms are maps that preserve all existing suprema. (Note that here a supremum is determined only up to order equivalence; that is, up to the relation  $(x\simeq y\iff x\leq y \text{ and } y\leq x)$.) $\mathbf{Ord}_{\mathrm{sup}}$
 contains $\mathbf{Pos}_{\mathrm{sup}}$ (partially ordered sets, with the same morphisms), and in both categories $A\rightsquigarrow y$ means that $\mathrm{sup} A$ exists and is isomorphic (or equal) to $y$. 
 
Let us give evidence that the above chain is isomorphic to
  $$\PP\text{-}\mathbf{ASpa}  \longrightarrow\PP\text{-}\mathbf{CFSpa}\longrightarrow\PP\text{-}\mathbf{CSpa}  \longrightarrow \mathcal{P}\text{-}\mathbf{Spa},$$
 Such evidence may be needed since, without reference to the relevant morphisms, already the paper \cite{Perry1976} claimed that $\PP$-spaces satisfying condition (C) are described as preordered sets. But the author may have based the claim in part on a faulty argument given in the paper \cite{Manes1974}, which assumes that $(X,\rightsquigarrow)\in  \PP\text{-}\mathbf{HSpa}$ may be ordered by  $(x\leq y\iff \{x,y\}\rightsquigarrow y)$; there is an obvious three-point counter-example, disproving the transitivity of this relation in general. 

After the $\PP$-for-$\UU$ exchange, Lemma \ref{StoneCechLemma} suggests that, given a $\PP$-space $(X,\rightsquigarrow)$, one should first close the relation $\rightsquigarrow$ under arbitrary  unions, to obtain the relation $\breve{\rightsquigarrow}$ on $\PP X$, defined by
$$A\breve{\rightsquigarrow}B\iff\exists  A_i\subseteq X, y_i\in X\, (i\in I): A=\bigcup_{i\in I}A_i, B=\{y_i\mid i\in I\}, A_i\rightsquigarrow y_i\, (i\in I).$$
The least equivalence relation $\sim$ on $\PP X$ containing $\breve{\rightsquigarrow}$ is still closed under arbitrary unions and is therefore a congruence relation on the free $\PP$-algebra $\PP X$. This allows us to define the $\PP$-algebra structure $c$ on $Q=\PP X/\!\sim$ with its projection $q:\PP X\to Q$ by
$$c(\{q(A_i)\mid i\in I\})=q(\bigcup_{i\in I}A_i).$$
Then, as expected, the map $X\to Q,\, x\mapsto q(\{x\})$, serves as the reflection $\beta_X$ of $(X,\rightsquigarrow)$ into $\mathbf{Sup}$, with $c$ assuming the role of $\bigvee$ in $Q$; the order in $Q$ is therefore described by ($q(A)\leq q(B)\iff q(A\cup B)=q(B)$) for all $A,B\subseteq X$, and we have
$$q(A)=\bigvee_{x\in A}q(\{x\}).$$
Let us now establish the isomorphism $\PP$-{\bf CSpa} $\lra$ {\bf Ord}$_{\mathrm{sup}}$. For a $\UU$-space $(X,\rightsquigarrow)$ to satisfy (C) means
$$A\rightsquigarrow y\iff \beta_X(A)=\beta_X(y)\iff q(A)=q(\{y\}),\quad(*)$$
for all $A\subseteq X,\,y\in X$. This allows us to define a preorder on $X$ by $(x\leq y\iff \beta_X(x)\leq \beta_X(y)),$ with respect to which one may describe the given convergence relation by
$$A\rightsquigarrow y\iff \sup A \text { exists in } X \text{ and } \sup A\simeq y.\;\qquad(**)$$
Conversely, given $(X,\leq)\in$ {\bf Ord}$_{\mathrm{sup}}$, we use $(**)$ to define the relation $\rightsquigarrow$ and obtain the $\UU$-space $(X,\rightsquigarrow)$. The task is now to show that $(*)$ holds, the non-trivial part of which is showing that $q(A)=q(\{y\})$ implies that $\sup A\simeq y$ holds in $(X,\leq)$. But this requires only an analysis of how the relation $\sim$ has been built, and the repeated application of the transitivity property: whenever $B=\bigcup_{i\in I}C_i$, with $\sup C _i$ existing in the preordered set $(X,\leq)$ for every $i\in I$, then $\sup B$ exists if, and only if, $\sup\{\sup C_i\mid i\in I\}$ exists, and the two suprema are order-equivalent.

We therefore have a bijective correspondence for the objects, which routinely extends to morphisms and therefore gives the desired isomorphism $\PP$-{\bf CSpa} $\lra$ {\bf Ord}$_{\mathrm{sup}}$. Since $(x\simeq y\iff \beta_X(x)=\beta_X(y))$ holds for all elements $x,y$ in a $T$-space $X$ satisfying (C), this isomorphism restricts to an isomorphism $\PP$-{\bf CFSpa} $\lra$ {\bf Pos}$_{\mathrm{sup}}$.

\medskip

(4) The focus in (3) is on the order-theoretic relevance of $\PP$-spaces, {\em i.e.}, on those $\PP$-spaces in which $A\rightsquigarrow y$ means being able to take a supremum of $A$ and obtaining $y$. In particular, up to order  equivalence, at most one element $y$ may be related to a given subset $A$. But there is a topological aspect of $\PP$-spaces that emerges when, in general, one forces potentially many elements $y$ to be related to $A$. More precisely, let the $\PP$-space $(X,\rightsquigarrow)$ satisfy the property
\begin{center}(Clo) \qquad $(A\rightsquigarrow y$\quad and\quad$A\subseteq B\subseteq X)\quad \Longrightarrow\quad B\rightsquigarrow y$  
\end{center}
and define the map $c:\PP X\to\PP X$\; by $cA=\{y\in X\mid A\rightsquigarrow y\})$. Then conditions (R) and (T) give us
$$ A\subseteq cA\quad\text{ and }\quad (B\subseteq cA\Longrightarrow cB\subseteq cA),$$
which characterize $c$ equivalently as an extensive, monotone and idempotent function, {\em i.e.}, as a closure operation on $X$, making $(X,c)$ a {\em closure space}. Conversely, having a closure operation $c$ on a set $X$  and setting $(A\rightsquigarrow y\iff y\in cA)$ makes $X$ a $\PP$-space satisfying (Clo). Under the emerging bijective correspondence, monotone maps $f:X\to Y$ of $\PP$-spaces correspond to continuous (= closure-preserving) maps: $f(c_XA)\subseteq c_Y(f(A))$. Hence, we may consider the category {\bf Clo} of closure spaces as fully embedded into $\PP$-{\bf Spa}. Of course, since a topology on $X$ may be described by a finitely additive closure operation on $X$ (so that also $c(A\cup B)=cA\cup cB,\; c\emptyset = \emptyset$ hold), the category {\bf Top} is fully (and coreflectively) embedded into {\bf Clo}; for a substantial generalization of these facts, see \cite{LaiTholen2017}.

 With Confession \ref{confession1}, one therefore has another chain of full embeddings
$$\UU\text{-}\mathbf{ASpa}  \longrightarrow\UU\text{-}\mathbf{Spa}\longrightarrow\mathbf{Clo}  \longrightarrow \mathcal{P}\text{-}\mathbf{Spa},$$ 
which seems to be noteworthy since it gives a direct way of moving from a $\UU$-space (convergence) structure $\to$ on $X$ to a $\PP$-space  (closure) structure $\rightsquigarrow$ on $X$ satisfying (Clo):
$$ A\rightsquigarrow y\iff\exists \,\mathfrak x\in \UU X:A\in  \mathfrak x\;\text{ and }\mathfrak x\to y.$$
Indeed, this equivalence means that $\rightsquigarrow$ is conveniently described as the relational composite of the membership relation followed by $\to$:
\begin{center}
$\xymatrix{\PP X\ar[rr]_{\in}\ar@/^0.8pc/[rrrr]^{\rightsquigarrow} && \UU X\ar[rr]_{\to} && X\\
}$	
\end{center}

\medskip
(5) The example $T=\UU$ has already been discussed in Section 2. There, however, we never took advantage of the unit interval $I$ which, thanks to Urysohn's Lemma, serves as a regular cogenerator of {\bf KHaus} (and of {\bf  Tych}). In fact, as already shown in \cite{Herrlich1968}, from the (dense, closed embedding)-factorization of the natural map $\theta_X$ of a topological space $X$ into the Tychonoff cube $I^{\Top(X,I)}$, one obtains rather easily the $\check{\mathrm C}$ech-Stone compactification $\beta_X:X\to \mathrm BX$---as a codomain restriction of $\theta_X$, since the diagonalization property of the factorization system lets any continuous map $f:X\to K$ into a compact Hausdorff space $K$ factor uniquely through $\beta_X$:
\begin{center}
	$\xymatrix{X\ar[rrrr]^{\theta_X}\ar[rrd]^{\beta_X}\ar[dd]_f &&&& I^{\mathbf{Top}(X,I)}\ar[dd]^{I^{\mathbf{Top}(f,I)}}\\
	&& \mathrm BX=\overline{\theta_X(X)}\ar@{>->}[rru]\ar@{-->}[lld]_{f^{\sharp}} &&\\
	K\;\ar@{>->}[rrrr]^{\theta_K} && && I^{\mathbf{Top}(K,I)}\\
	}$
\end{center}

\begin{rems}
\rm 
(1)	The membership relation mentioned in Example \ref{examples}(4)  is (the $X$-component of) a {\em} lax natural transformation $\in:\hat{\PP}\to\hat{\UU}$ of the lax extensions of the monads $\PP$ and $\UU$, from $\Set$ to the category {\bf Rel} of sets and relations (see Remarks \ref{laxUalgebras}(2)), in the sense that, for all $R\subseteq X\times Y$, we have the laxly commuting diagrams
	\begin{center}
	$\xymatrix{\PP X\ar[d]_{\in_X}^{\quad\;\;\subseteq}\ar[r]^{\hat{\PP}R} & \PP Y\ar[d]^{\in_Y} & X\ar[d]_{\{-\}}^{\;=}\ar[rd]^{\dot{(-)}} & & \PP\PP X\ar[r]^{\hat{\PP}\in_X}\ar[d]_{\bigcup} & \PP\UU X\ar[r]^{\in_{\UU X}} & \UU\UU X\ar[d]^{\Sigma}_{\subseteq\qquad\qquad\;\;}  \\
	  \UU X\ar[r]_{\hat{\UU}R} & \UU Y  & \PP X\ar[r]_{\in_X} & \UU X & \PP X\ar[rr]_{\in_X} & & \UU X\\
	     }$	
	\end{center}
	(2) The paper \cite{Manes20210} offers a rich supply of further topologically interesting monads.
\end{rems}

\end{exas}

\section{$T$-spaces in a category}

We now show that Confession \ref{confession2} may actually be extended from the category $\Set$ to a fairly arbitrary category $\CC$ that comes equipped with a monad $T=(T,\eta,\mu)$ on $\CC$. Concretely, let us assume that
\begin{itemize}
\item $\CC$ is a complete and well-powered category (so that it admits all small-indexed limits and that each of its objects has a representative {\em set} of subobjects);
\item every morphism in $\CC$ factors into a regular epimorphism followed by a monomorphism (epimorphisms are regular when they appear as coequalizers);
\item the endofunctor $T:\CC\to \CC$ of the given monad maps regular epimorphisms to regular epimorphisms.
\end{itemize}
Note that these conditions are certainly satisfied for $\CC=\Set$, including the last one: with Choice, surjections are retractions and therefore preserved by any (endo)functor. In what follows, not all of these conditions will be needed right away; in particular, the third one is not used before Theorem \ref{AlgreflinSpa}. The conditions are chosen to make sure that the entire program of Section 2 carries over easily to the current setting.

\begin{defi}\label{TspaceinCdef}
\rm 
A $T$-{\em space} $X=(X,A,d,c)$ in $\CC$ is given by $\CC$-objects $X$ (to be thought of as the {\em object of points} of the $\T$-space) and  $A$ (the {\em object of wiggly arrows} of the $T$-space) and morphisms $d: A\to TX$ and $c: A\to X$ (to be thought of as assigning to a wiggly arrow its {\em domain} and {\em codomain}, respectively), such that
the pair $(d,c)$ is jointly monic; equivalently, the induced morphism $\langle d,c\rangle :A\to TX\times X$ is monic (and thus represents the {\em convergence relation} of the $T$-space); furthermore,
\begin{itemize}
\item[(R)] there is a morphism $i:X\to A$ with $d\cdot i=\eta_X$ and $c\cdot i=1_X$ (called the {\em insertion of wiggly identity arrows});
\item[(T)] there is a morphism $m: TA\times_{TX}A\to A$ with $d\cdot m=\mu_X\cdot Td\cdot d'$ and $c\cdot m=c\cdot c'$ (called the multiplication or {\em composition of wiggly arrows});
\begin{center}
	$\xymatrix{&& X\ar[lld]_{\eta_X}\ar[d]^i\ar[rrd]^{1_X} &&\\
	TX && A\ar[ll]_d\ar[rr]^c && X\\
	TA\ar[u]^{\mu_X\cdot Td}\ar[rrd]_{Tc} && TA\times_{TX}A\ar[ll]_{d'}\ar[u]_m\ar[rr]^{c'} && A\ar[u]_c\ar[lld]^d\\
	&& TX &&\\
	}$
\end{center}
here, the pullback object $TA\times_{TX}A$ of the pair $(Tc,d)$ with its projections $d',c'$ should be thought of as the object of {\em composable pairs of wiggly arrows} of the $T$-space.
\end{itemize}
\end{defi}

\begin{rem}
\rm 
 Note that the morphisms $i$ and $m$ are uniquely determined since the pair $(d,c)$ is monic. If one removes the monic condition, one arrives at the notion of (small internal) $T$-{\em category} in $\CC$, by considering the morphisms $i$ and $m$ as parts of the structure, and by imposing the expected identity and associativity conditions on them---which, however, may be considered to be a bit cumbersome in general. But in our setting they come for free. For details, we refer to Burroni's original paper \cite{Burroni1971} and the more recent paper \cite{TholenYeganeh2021}.
\end{rem}

\begin{defi}
\rm For $T$-spaces $X=(X,A,d_X,c_X),\; Y=(Y,B,d_Y,c_Y)$, a morphism $f:X\to Y$ in $\CC$ is called {\em monotone}, if there is a morphism $\overline{f}: A\to B$ in $\CC$ with $d_Y\cdot \overline{f}=Tf\cdot d_X$ and $c_Y\cdot \overline{f}=f\cdot c_X$ (which makes $\overline{f}$ determined by $f$).
\begin{center}
$\xymatrix{TX\ar[d]_{Tf} && A\ar[ll]_{d_X}\ar[rr]^{c_X}\ar[d]^{\overline{f}} && X\ar[d]^f\\
TY && B\ar[ll]^{d_Y}\ar[rr]_{c_Y} && Y\\
}$	
\end{center}
The $T$-spaces in $\CC$ together with the monotone morphisms form the category
\begin{center}
	$T$-{\bf Spa}$(\CC)$.
\end{center}
\end{defi}

\begin{rems}
\rm 
(1) In case $\CC=\Set$, up to categorical equivalence $T$-{\bf Spa}$(\CC)$ is the category $T$-{\bf Spa} of Section 3. Indeed, given $(X,A,d,c)\in T$-{\bf Spa}$(\Set)$, let
$$C=\langle d,c\rangle(A)\subseteq TX \times X.$$
\begin{center}
$\xymatrix{A\;\ar@{>->}[rd]^{\langle d,c\rangle}\ar[dd]_{\cong} & && TA\ar[dd]_{\cong}\ar@{->>}[dr]\ar[drr]^{\langle Td,Tc\rangle} & & \\
& TX\times X && & \hat{C}\;\ar@{>->}[r] & TTX\times TX\\
C\;\ar@{>->}[ru]_{\langle\pi_1,\pi_2\rangle} & && TC\ar@{->>}[ru]\ar[rru]_{\langle T\pi_1,T\pi_2\rangle} & &\\
}$	
\end{center}
Then $\hat{C}=\langle T\pi_1,T\pi_2\rangle(TC)=\langle Td,Tc\rangle(TA)\subseteq TTX\times TX$, and conditions (R),
(T) as stated in this section are easily seen to translate to conditions (R),
(T) as stated in Sections 2 and 3. Conversely, given $(X,C)$ in $T$-{\bf Spa}, one has $(X,C,\pi_1,\pi_2)$ in $T$-{\bf Spa}$(\Set)$. Obviously, these object assignments preserve the monotonicity of maps.

(2) It is important to note that the above diagrams make sense {\em without } the provision $\CC=\Set$. Indeed, taking advantage of the (regular epi, mono)-factorization system in our general category $\CC$, we can define the relation $\hat{C}$ as in the $\Set$-case. This means that the relational description given for $\CC=\Set$ translates to the general case. In particular, if $\CC$ is a variety as in universal algebra, like that of groups, rings, $R$-modules, {\em etc.}, we may describe the objects of $T$-{\bf Spa}$(\CC)$ in terms of congruence relations, by just following the lead of the relational description in the $\Set$-case.

(3) For $T$ the identity monad on the category $\CC$, $T$-spaces in $\CC$ are {\em internal preorders} of $\CC$. Therefore, in the papers \cite{Burroni1971}, \cite{TholenYeganeh2021}, also for a general monad $T$ the category $T$-{\bf Spa}$(\CC)$ is denoted by {\bf Ord}$(T)$.
\end{rems}

Let	us first confirm that Facts \ref{USpatopological}(1),(2) remain valid in the current categorical context, as already claimed in \cite{TholenYeganeh2021}. Since some readers may find it difficult to extract these facts from the categorical setting of that paper, we sketch their easy proofs here.
\begin{pro}\label{SpaTtopological}
The forgetful functor 	
 $T$-{\bf Spa}$(\CC)\lra\CC, \;(X,A,d,c)\longmapsto X,$ is topological. It therefore admits both, a left adjoint and a right adjoint. Like $\CC$, $T$-{\bf Spa}$(\CC)$ is complete and well-powered and  has a (regular epi, mono)-factorization system, preserved by the forgetful functor. If $\CC$ is also cocomplete, so is $T$-{\bf Spa}$(\CC) $.
 \end{pro}
 
 \begin{proof}
 We consider a $\CC$-object $X$ and a possibly large family $f_i:X\to Y_i\;(i\in I)$ of morphisms 
 whose codomain carries a $T$-space structure: $Y_i=(Y_i,B_i,d_i,c_i)$; we must find the initial $T$-structure $(X,A,d,c)$, making all morphism $f_i$ monotone. This may be done by first forming the pullback $r_i:A_i\to TX\times X$ of $\langle d_i,c_i\rangle:B_i\to TY_i\times Y_i$ along the morphism $Tf_i\times f_i$, for every $i\in I$. As a pullback of a monomorphism, every $r_i$ is monic. Since $\CC$ is well-powered, we can choose a representative set of non-isomorphic subobjects of $TX\times X$ amongst the morphisms $r_i\;(i\in I)$. Hence, without loss of generality, we may assume that $I$ is small, {\em i.e.}, a set, and we can then form the intersection $r:A\to TX\times X$ (= wide pullback) of that family; it comes with monomorphisms $t_i:A\to  A_i$. Considering the components $d,c$ of $r$ and setting $\overline{f_i} = f_i'\cdot t_i$ (with $f_i': A_i\to B_i$ a pullback projection), one routinely proves that $(X,A,d,c)$ has the required lifting property.
 \begin{center}
 	$\xymatrix{A\ar[dd]_{\overline{f_i}}\ar[rd]^{t_i}\ar[rr]^{r=\langle d,c\rangle} && TX\times X\ar[dd]^{Tf_i\times f_i}\\
 	& A_i\ar[ru]^{r_i}\ar[dl]_{f_i'} & \\
 	B_i\ar[rr]^{\langle d_i,c_i\rangle} && TY_i\times Y_i
 	}$
 \end{center}
This confirms the topologicity of the forgetful functor, and the other assertions are generally valid consequences of this fact (see \cite{AHS},\cite{MonTop}).
 \end{proof}

\begin{pro}\label{AlgfullyinSpa}
The functor $\CC^T\longrightarrow T$-{\bf Spa}$(\CC),\; (X,c)\longmapsto	 (X,TX,1_{TX},c)$, embeds the category of $T$-algebras fully into $T$-{\bf Spa}$(\CC)$ and renders it equivalent to the full subcategory $T$-{\bf ASpa}$(\CC)$, formed by those $T$-spaces $(X,A,d,c)$ for which
\begin{itemize}
\item[{\em (A)}] $d:A\to TX$ is an isomorphism in $\CC$.	
\end{itemize}
\end{pro}

\begin{proof} Given a $T$-algebra $(X,c)$, having put $A=TX$ and $d=1_{TX}$, and using the notation
of Definition \ref{TspaceinCdef}, we see immediately that we may take $TTX$ to assume the role of $TA\times_{TX} A$, with $d'=1_{TTX}$ and $c'=Tc$. Then $i=\eta_X$ and $m=\mu_X$ are (the only choices) of morphisms to make $(X,TX,1_{TX},c)$	a $T$-space. Furthermore, a $\CC$-morphism $f:X\to Y$ of $T$-algebras becomes monotone precisely when it is a $T$-homomorphism. The category equivalence $\CC^T\simeq T$-{\bf ASpa}$(\CC)$ follows routinely. 
\end{proof}

\begin{theorem}\label{AlgreflinSpa}
$T$-{\bf ASpa}$(\CC)$ is reflective in $T$-{\bf Spa}$(\CC)$.
\end{theorem}

\begin{proof}
It suffices to show that the functor	 $\CC^T\longrightarrow T$-{\bf Spa}$(\CC)$ of Proposition
\ref{AlgfullyinSpa} has a left adjoint. First recall that, since $T$ is assumed to preserve regular epimorphisms, with $\CC$ also $\CC^T$ has (regular epi, mono)-factorizations, which the forgetful functor $\CC^T\to \CC$ preserves. Moreover, every regular epimorphism with domain $X$ is determined by its kernel pair which, in turn, is determined by a (regular) subobject of $X\times X$; consequently, as $\CC$ is well-powered, $\CC$ must be cowell-powered with respect to regular epimorphisms, and the same is true for $\CC^T$.

Now, in order to find a universal arrow for a $T$-space $(X,A,d,c)$, one considers a representative set $q_i: (TX,\mu_X)\to (Q_i,b_i)\;(i\in I)$ of those regularly epic $T$-homomorphisms for which $q_i\cdot\eta_X:X\to Q_i$ is monotone; the monotonicity means that
$$b_i\cdot Tq_i\cdot T\eta_X\cdot d=q_i\cdot \eta_X\cdot c$$
holds for all $i\in I$. With $\overline{b}$ denoting the $T$-algebra structure of $\prod_{i\in I}Q_i$, we have the induced $T$-homomorphism 
$f:(TX,\mu_X)\to(\prod_{i\in I}Q_i,\overline{b})$ which satisfies $\overline{b}\cdot Tf= f\cdot \mu_X$ and 
$\pi_i\cdot \overline{b}=b_i\cdot T\pi_i$, with product projections $\pi_i\;(i\in I)$. We claim that, with the (regular epi, mono)-factorization
$$\xymatrix{(TX,\mu_X)\ar[r]^e & (Q,b)\ar[r]^{m\quad} & (\prod_{i\in I} Q_i,\overline{b})\\
}$$
of $f$ in $\CC^T$, the morphism $e\cdot \eta_X:X\to Q$ is monotone. Indeed, since the family $\pi_i\cdot m\, (i\in I)$ is collectively monic, from
\begin{align*}
\pi_i\cdot m\cdot b\cdot Te\cdot T\eta_X\cdot d &=\pi\cdot \overline{b}\cdot Tm\cdot Te\cdot Te\cdot T\eta_X\cdot d\\
& =	\pi_i\cdot\overline{b}\cdot Tf\cdot T\eta_X\cdot d\\
& = b_i\cdot T\pi_i\cdot Tf\cdot T\eta_X \cdot d\\
& = b_i\cdot Tq_i\cdot T\eta_X\cdot d\\
& = q_i\cdot \eta_X\cdot c\\
& = \pi_i\cdot m\cdot e\cdot\eta_X\cdot c
\end{align*}
for all $i\in I$ we can deduce\; $b\cdot Te\cdot T\eta_X\cdot d= e\cdot\eta_X\cdot c$, which means that $e\cdot\eta_X:X\to Q$ is monotone. In fact, $e\cdot\eta_X$ is the desired universal arrow, as one routinely verifies: given any monotone morphism $g:X\to Y$ to a $T$-algebra $(Y,a)$, form the (regular epi, mono)-factorization $g^{\sharp}=m'\cdot e'$ of $g$'s  mate $g^{\sharp}=a\cdot Tg: TX\to Y$. Then $e'\cdot \eta_X$ inherits its  monotonicity from $f$, so that $e'$ must be isomorphic to one of the $q_i$s, each of which factors through $e$. Hence, for some isomorphism $t$ and $j\in I$, one readily sees that $h=m'\cdot t\cdot \pi_j\cdot m:Q\to Y$ is the only $T$-homomorphism satisfying $g=h\cdot e\cdot\eta_X$.
\end{proof}


Like in the case $\CC=\Set$, in what follows we denote the reflection morphism of $X\in T$-{\bf Spa}$(\CC)$ to $T$-{\bf ASpa}$(\CC)$ by $\beta_X:X\to\mathrm BX$ also in the general case. We extend the arrangements of Definition \ref{Ucomplreg} to the general case in an obvious manner:

\begin{defi}
\rm 
A $T$-space $X=(X,A,d,c)$ is said to satisfy condition
\begin{itemize}
\item[(K)] if $d:A\to TX$ is a split epimorphism in $\CC$;
\item[(H)] if $d:A\to TX$ is a monomorphism in $\CC$;
	\item[(C)] if $X$ carries the cartesian structure with respect to $\beta_X$ and the topological functor $T\text{-{\bf Spa}}(\CC)\to\CC$ (see Proposition \ref{SpaTtopological}); 
	\item[(F)] if $\beta_X$ is a monomorphism in $T\text{-{\bf Spa}}(\CC)$;	
	\item[(CF)] if $X$ satisfies (C) and (F); that is: if $\beta_X$ is a regular	 monomorphism in $T\text{-{\bf Spa}}(\CC)$.
		\end{itemize}
For {\bf X} = {\bf K}, {\bf H}, {\bf C}, {\bf F}, and {\bf CF}, this defines the full subcategories $$T\text{-{\bf XSpa}}(\CC)\hookrightarrow T\text{-{\bf Spa}}(T).$$  
\end{defi}

\begin{rem}
\rm 
An explanation for the definition of condition (K) seems to be called for. Since, under our general hypotheses, regular epimorphisms in $\CC$ appear to assume the role of surjections in $\Set$, why not ask that $d$ be a regular epimorphism in $\CC$, rather than a split epimorphism? In fact, the case for split epimorphism is compelling when one looks at $\langle d,c\rangle: A\to TX\times X$ as an arrow $A:TX\to X$ in the 2-category {\bf Rel}$(\CC)$, which has the same objects as $\CC$, and where the relational composite $S\circ R$ of the arrow $R:X\to Y$ with $S:Y\to Z$ is obtained as the regular (!) image of the obvious morphism $R\times_Y S\to X\times Z$ in $\CC$; the diagonal relation $\Delta_X$ serves as an identity arrow, and 2-cells are given by preorder.

Now, as in every 2-category, by (Lawvere's) definition, our arrow $A: TX\to X$ is a {\em map} in {\bf Rel}$(\CC)$
if $A$ is left adjoint to its converse $A^{\circ}$ (given by $\langle c,d\rangle$); that is, if we have:

 One can now proceed and show routinely
  that (K') holds precisely when $d$ is a split epimorphism in $\CC$, and that (H') equivalently means that $d$ is a monomorphism in $\CC$.
\end{rem}

Let us confirm that Proposition \ref{lowhangingfruit} and Corollary \ref{Hausdorffreflection}}  remain true in the current context:
\begin{theorem}\label{lowhangingfruitgeneralized}
Let $f:X\to Y$ and $f_i:X\to Y_i\; (i\in I)$	 be monotone morphisms of $T$-spaces (with $I$ possibly large). Then:
\begin{itemize}
\item[{\em (1)}]  If $X$ satisfies {\em (K)}, so does $Y$, provided that either {\em (a)} $f$ is a split epimorphism in $\CC$, or {\em (b)} $f$ is a regular epimorphism in $\CC$ and $Y$ satisfies {\em (H)}.
\item[{\em(2)}] {\em (Tychonoff's Theorem)} The product of a family of $T$-spaces satisfying {\em (K)} also satisfies {\em (K)}.
\item[{\em (3)}] If the family $(f_i)_{\in I}$ is jointly monic in $\CC$, and if every $Y_i$ satisfies {\em (H)}, then so does $X$.
\item[{\em (4)}] $T\text{-{\bf HSpa}}(\CC)$ is regular-epi-reflective in $T\text{-{\bf Spa}}(\CC)$, and $T\text{-{\bf ASpa}}(\CC)$ is regular-epi-reflective in $T\text{-{\bf KSpa}}(\CC)$.\end{itemize}	
\end{theorem}

\begin{proof}
(1) By hypothesis, we have $s:A\to TX$ with $d_X\cdot s=1_{TX}$. Under assumption (a), one also has $g:Y\to X$ with $f\cdot g=1_Y$, which makes $d_Y$ have a section as well:
$$d_Y\cdot\overline{f}\cdot s\cdot Tg=Tf\cdot d_X\cdot s\cdot Tg =Tf\cdot Tg=1_{TY}.$$
Under assumption (b), we note that, by our general assumptions, with $f$ also $Tf:TX\to TY$ is a regular epimorphism. Hence, whenever $Tf\cdot a=Tf\cdot b$ in $\CC$, we have
$$d_Y\cdot \overline{f}\cdot s\cdot a= Tf\cdot d_X\cdot s\cdot a= Tf\cdot d_X\cdot s\cdot b=d_Y\cdot \overline{f}\cdot s\cdot b, $$
which implies $\overline{f}\cdot s\cdot a=\overline{f}\cdot s\cdot b$ since $Y$ satisfies (H). Consequently, $\overline{f}\cdot s$ factors through $Tf$, by a morphism $t:TY\to B$ which turns out to serve as a section for $d_Y$. Indeed, this follows from $d_Y\cdot t\cdot Tf=d_Y\cdot\overline{f}\cdot s=Tf\cdot d_X\cdot s= Tf$ since $Tf$ is epic.
\medskip

(2)  Let $(X,A,d,c)$ be the product of $(X_i,A_i,d_i,c_i)\;(i\in I)$ in $T\text{-{\bf Spa}}(\CC)$, with product projections $p_i$ and every $X_i$ satisfying (K). As the product is constructed by initially lifting the product of the underlying $\CC$-objects, with Proposition \ref{SpaTtopological} one sees that $A$ is obtained as the limit in $\CC$ of the diagram
\begin{center}
$\xymatrix{TX\ar[d]_{Tp_i} && X\ar[d]^{p_i}\\
TX_i & A_i\ar[l]_{d_i}\ar[r]^{c_i} & X_i\;,\\
}$	
\end{center}
with limit projections $(d:A\to TX,\; \overline{p_i}:A\to A_i,\; c:A\to X)$. Now, having sections $s_i$ of $d_i$ for every $i\in I$, the product property of $X$ in $\CC$ gives us the morphism 
$h: TX\to X$ with $p_i\cdot h= c_i\cdot s_i\cdot Tp_i\; (i\in I)$. Since also $d_i\cdot s_i\cdot Tp_i= Tp_i\cdot 1_{TX}\; (i\in I) $, the limit property makes the morphisms ($1_{TX},\; s_i\cdot Tp_i:TX\to A_i,\; h)$ factor through the limit projections $(d, \overline{p_i}, c)$, by a morphism $s:TX\to A$ which, in particular, must satisfy $d\cdot s=1_{TX}$ and, hence, makes $X$ satisfy (K).

\medskip
(3) With $X=(X,A,d,c)$ and $Y_i=(Y_i,B_i,d_i,c_i)$ and every $d_i\;(i\in I)$ monic in $\CC$, assume $d\cdot a=d\cdot b$. Composition with $Tf_i$ gives $d_i\cdot \overline{f_i}\cdot a=d_i\cdot\overline{f_i}\cdot b$, whereupon first $\overline{f_i}\cdot a=\overline{f_i}\cdot b$ and then $f_i\cdot c\cdot a=f_i\cdot c\cdot b\;(i\in I) $ follow. As $(f_i)_{i\in I}$ is monic, we conclude $c\cdot a=c\cdot b$ which, in conjunction with $d\cdot a=d\cdot b$, gives $a=b$.

(4) The first assertion is an easy consequence of the stability property (3). In Theorem \ref{complregrefl}(4) we outline the proof in the completely analoguous situation where (H) is replaced by (F). To obtain the second assertion, just note that, by (1), version (b), the reflector of the first assertion will preserve property (K).
\end{proof}

Next we confirm that Theorem \ref{lowhangingfruitbis} and Corollary \ref{lowcoro}  remain true in the current context:

\begin{theorem}\label{complregrefl}
 Let $f_i:X\to Y_i\;(i\in I)$ be (a possibly large) family of monotone morphisms of $T$-spaces with common domain $X$. Then:
\begin{itemize}
\item[{\em(1)}] If $X$ carries the initial structure with respect to $(f_i)_{i\in I}$ and the topological functor $T\text{-{\bf Spa}}(\CC)\to\CC$, and if every $Y_i$ satisfies {\em (C)}, then so does $X$. 
\item[{\em(2)}] If $(f_i)_{i\in I}$ be collectively monic, and if every $Y_i$ satisfies {\em (F)}, then so does $X$.
\item[{\em (3)}]
$T\text{-{\bf CSpa}}(\CC)$ is simultaneously epi- and mono-reflective in $T\text{-{\bf Spa}}(\CC)$. 
\item[{\em (4)}] $T\text{-{\bf FSpa}}(\CC)$ is regular-epi-reflective in $T\text{-{\bf Spa}}(\CC)$.	
\item[{\em (5)}] $T\text{-{\bf CFSpa}}(\CC)$ is epi-reflective in $T\text{-{\bf Spa}}(\CC)$.	 
\end{itemize}	
\end{theorem}

\begin{proof}
	Since the proofs for (1) and (2) proceed exactly as in the special case $\CC=\Set$ and $T=\UU$, we restrict ourselves to indicating only how (3)-(5) follow quite generally from (1) and (2) (for more details we refer to \cite{AHS}, \cite{MonTop}).
	
	(3) Given a $T$-space $X$, we consider the (large) family $f_i:X\to Y_i\; (i\in I)$ of {\em all} monotone morphisms with domain $X$ and codomain satisfying (C), and let $X'$ be the $T$-space that has the same underlying $\CC$-object as $X$ but carries the initial $\T$-structure with respect to $(f_i)_{i\in I}$. Then the identity morphism $X\to X'$ is monotone and serves as a reflection of $X$ into $T\text{-{\bf CSpa}}(\CC)$. It is obviously both epic and monic.
	
	(4) Given $X$, we now let $f_i:X\to Y_i\; (i\in I)$ be the family of all monotone morphisms with codomains lying in 	$T\text{-{\bf FSpa}}(\CC)$. It will suffice to show that this family (and, in fact, any family in $T\text{-{\bf Spa}}(\CC)$ with fixed domain) factors as $f_i=m_i\cdot e\; (i\in I)$, with a regular epimorphism $e$ and a monic family $(m_i)_{i\in I}$, as follows. First, for every $i\in I$ consider the (regular epi, mono)-factorization $f_i=u_I\cdot e$ of $f_i$. Since $\CC$ (and, hence, $T\text{-{\bf FSpa}}(\CC))$ is cowell-powered with respect to regular epimorphisms, there is a representative set of non-isomorphic $e_i$s; for notational simplicity, we assume that the $e_i$s  are already non-isomorphic and form the (regular epi, mono)-factorization $h=m\cdot e$ of the morphism $h:X\to\prod_{i\in I} Y_i$ that is induced by the family $(e_i)_{i\in I}$. With another (regular epi, mono)-factorization $h=m\cdot e$ we can finish the proof, by setting $m_i=u_i\cdot \pi_I\cdot m \;(i\in I)$, where $\pi_i$ is a product projection.
	
	(5) First apply the reflector of (4) and then the reflector of (3).
		\end{proof} 
		
From Proposition \ref{SpaTtopological} and Theorem	\ref{complregrefl}(1) we conclude immediately:		
		
\begin{cor}
The forgetful functor $T\text{-{\bf CSpa}}(\CC)\to\CC$ is topological, with initial liftings formed like in $T\text{-{\bf Spa}}(\CC)$.
\end{cor}

\begin{exa}
\rm 
	Let $M=(M,e,m)$ be a monoid. An $M$-{\em set} $X$ is a set $X$ that comes with a (multiplicatively written) left action $M\times X\to X$ that is unitary and associative ($ex=x$ and $t(sx)=(ts)x$ for all $s,t\in M$ and $x\in X$). The structure is equivalently described by a homomorphism $M\to \Set(X,X)$ of monoids, or by a functor $M\to\Set$ that maps the only object of the category $M$ to the set $X$; we will call that functor $X$ again and write its value at the morphism $s$ of $M$ simply as the ``translation'' map $s:X\to X,\,x\mapsto sx$. A morphism $f:X\to Y$ of $M$-sets is an equivariant map (so that $f(sx)=s(fx)$ for all $s\in M,\,x\in X$). This defines the category
	$$\Set^M=[M,\Set]$$
	of $M$-sets, which may be thought of as both, the Eilenberg-Moore category of the monad $M\times (\text{-})$ induced by $M$, and the functor category of $\Set$-valued functors on $M$. Under the latter interpretation, we may trade $\Set$ for any category $\CC$ and obtain the category of  $M$-{\em objects} $X$ in $\CC$, {\em i.e.}, of $\CC$-objects $X$ that come with endomorphisms $s:X\to X\,(s\in M)$, such that $e$ is the identity morphism on $X$ in $\CC$ and the monoid product $ts$ is also the composite of the morphisms of $s$ and $t$ in $\CC$. For example,
	$$\mathbf{Top}^M=[M,\mathbf{Top}]$$
	is the category of $M$-topological spaces; its objects are topological spaces that carry an $M$-set structure making the translation maps continuous, and its morphisms are continuous equivariant maps.
	
	\medskip
	Now {\em assume $M$ to be commutative}; this makes the translation maps of any $M$-set equivariant.  The paper \cite{AdamakSousa2019} gives the general categorical reasons for why the ultrafilter monad $\UU$ may be ``lifted'' from $\Set$ to $\Set^M$, to become a monad $\UU^M$ on $\Set^M$, with the same underlying sets as $\UU$, a fact that one may quite easily check directly, as follows. First observe that, with $X$ also the power-set $\PP X$ becomes an $M$-set, when we define $sA$ (for $s\in M$ and $A\subseteq X $) as the inverse image of $A$ under the translation map by $s$:
	$$sA=\{x\in X\mid sx\in A\}.$$
	Consequently, also $\PP\PP X$ becomes an $M$-set, and we can easily see that $\UU X$ is an $M$-subset; indeed, for $\mathfrak x\in \UU X$,
	$$s\mathfrak x=\{A\subseteq X\mid sA\in \mathfrak x\}$$
	is again an ultrafilter on $X$. Furthermore, for every equivariant map $f:X\to  Y$, the map $\UU f$ is also equivariant, and so are the monad-structure maps $\dot{(\text{-})}$ and $\Sigma_X$. This gives the monad $\UU^M$ on $\Set^M$ which makes the diagram
	\begin{center}
		$\xymatrix{\Set^M\ar[r]^{\UU^M}\ar[d] & \Set^M\ar[d]\\
		\Set\ar[r]^{\UU} & \Set\\
		}$
	\end{center}
	commute. A $\UU^M$-space structure on an $M$-set $X$ is given by an $M$-subset $C\subseteq \UU X\times X$ satisfying (R) and (T); that is: by an ultrafilter convergence relation $\rightsquigarrow$ making $X$ a topological space such that all translation maps become continuous:
	$$\mathfrak x\rightsquigarrow y\quad\Longrightarrow\quad s\mathfrak x\rightsquigarrow sy,$$
	for all $\mathfrak x\in \UU X,\, y\in X,\, s\in M$. This leads us to
	$$ (\Set^M)^{\UU^M}\cong \UU^M\text{-}\mathbf{ASpa}(\Set^M)\cong \mathbf{KHaus}^M\hookrightarrow \mathbf{Top}^M\cong \UU^M\text{-}\mathbf{Spa}(\Set^M)\,.$$
	We need to describe the reflector of this full embedding, in order to be able to characterize condition (C) for $M$-topological spaces. Fortunately, the (huge) 2-category $\mathbf{CAT}$ lets us conclude that the adjunction $\mathrm{B}\dashv \mathrm{Inc}:\mathbf{KHaus}\to\mathbf{Top}$ gives
	\begin{center}
	$\xymatrix{[M,\mathbf{KHaus}]\ar@/_0.5pc/[rr]_{\quad[M,\mathrm{Inc}]} & \bot & [M,\mathbf{Top}]\ar@/_0.5pc/[ll]_{\quad [M,\mathrm B]}.\\
	}$ 	
	\end{center}
	This means that every translation $s:X\to X$ of an $M$-topological space $X$ must give the translation $\mathrm Bs:\mathrm BX\to \mathrm BX$, {\em i.e.,} the continuous map determined by making the diagram
	
	\begin{center}
	$\xymatrix{X\ar[r]^s\ar[d]_{\beta_X} & X\ar[d]^{\beta_X}\\
	\mathrm BX\ar[r]^{\mathrm Bs} & \mathrm BX\\
	}$	
	\end{center}
commute. In other words: we can make the compact Hausdorff space $\mathrm BX$ a topological $M$-space by putting $$s\mathfrak z = (\mathrm Bs)(\mathfrak z),$$
for all $s\in M,\,\mathfrak z\in\mathrm BX$,
and this makes the map $\beta_X$ equivariant. In this way we see that $\beta_X$ serves as a reflection of $X$ into $\mathbf{KHaus}^M$.
\end{exa}
Let us finally observe that, for a morphism $f:X\to Y$ in $\mathbf{Top}^M$, initiality with respect to the (topological) functor $\mathbf{Top}^M\to\Set^M$ is already characterized by initiality with respect to $\mathbf{Top}\to\Set$ (since, when {\em defining} $\mathfrak x\rightsquigarrow x$ by $\UU f(\mathfrak x)\rightsquigarrow f(x)$, the new topological structure on the $M$-set $X$ makes its translations	continuous).	 Applying this observation to $\beta_X$ we see that, for a topological $M$-space, condition (C) just means that (C) holds for its underlying topological space, that is:
$$\UU^M\text{-}\mathbf{CSpa}(\Set^M)\cong\mathbf{CReg}^M.$$

\begin{rems}
\rm 
	(1) The category $\mathbf{KHaus}^M$ appears explicitly already in \cite{Manes1969} (Example 7.2), as an example of a monadic category over $\Set$ that is built from the two interacting $\Set$-monads $M\times (\text{-})$ and $\UU$. 
	
	(2) The general results of the papers \cite{Leinster2013} and \cite{AdamakSousa2019}, and the examples mentioned in them, suggest that so-called {\em codensity monads} offer themselves most naturally for the investigation of the properties (X) considered in Section 4. The double-dualization monad of the category $K$-{\bf Vec} is such a monad, with the elements of the double-dual $X^{**}$ of a $K$-vector space $X$ assuming the role of ultrafilters. The Eilenberg-Moore algebras of this monad have been identified in \cite{Leinster2013} as the so-called linearly compact (Hausdorff) $K$-vector spaces.
	
	(3) It is well known that the procedure for the construction of the $\check{\mathrm C}$ech-Stone compactification as given in Example \ref{examples}(5) may be replicated for the construction of the Bohr compactification of a topological abelian group. Indeed, with the regular cogenerator $\mathbb T= \mathbb R/\mathbb Z$ of the category {\bf KHausAb} of compact Hausdorrf abelian groups and their continuous homomorphisms (also known as the dualizing object of the Pontryagin duality) at one's disposal, one may obtain the Bohr compactification of a topological abelian group $X$ as follows: replacing the unit interval by the circle group $\mathbb T$, just consider the closure of the image of the natural map $X\to \mathbb T^{\mathbf{TopAb}(X,\mathbb T)}$. Taking $X$ to be discrete, one obtains the (endofunctor of a) monad $\UU^{\mathbf{Ab}}$ on $\mathbf{Ab}$, for which the category $\UU^{\mathbf{Ab}}\text{-}\mathbf{Spa}(\mathbf{Ab})$ needs to be investigated.	
	Since, by a well-known, but non-trivial, result of harmonic analysis, a homomorphism $f:X\to Y$ of locally compact abelian (lca) groups is continuous, as soon as $\chi f$ is continuous for all (characters) $\chi\in\mathbf{TopAb}(Y,\mathbb T)$, we conjecture that all lca groups lie in $\UU_{\mathbf{Ab}}\text{-}\mathbf{CFSpa}(\mathbf{Ab})$, but must leave any elaboration of this conjecture for future work.
	\end{rems}


\section{External properties of complete regularity}
	
Continuing to work in a category $\CC$ equipped with a monad $T$, under the same general assumptions as in Section 4, we have the diagram 
\begin{center}
$\xymatrix{T\text{-{\bf ASpa}}(\CC)\ar[rd]_U\ar@/_0.5pc/[rr]_J^{\bot} &  & T\text{-{\bf CSpa}}(\CC)\ar[ld]^V\ar@/_0.5pc/[ll]_{\mathrm B}\\
& \CC & \\
}$
	\end{center}
	where $U,V$ are forgetful functors and $J$ is a full embedding with reflector $B$, the existence of which is guaranteed by Theorem \ref{AlgreflinSpa}. Since $VJ=U$, we may consider $J$ as a morphism $J:U\to V$ in the (huge) category {\bf CAT}/$\CC$ of categories over $\CC$ (see Remarks \ref{finalrems}(2)). It is noteworthy that $J$ embeds the {\em monadic} (and, hence, ``algebraic'') category $T\text{-{\bf ASpa}}(\CC)$ over $\CC$ reflectively into the {\em topological} category $T\text{-{\bf CSpa}}(\CC)$ over $\CC$ which, as such, is in particular {\em fibred} over $\CC$, {\em i.e.}, $V$ is a Grothendieck fibration (\cite{MonTop}). 
	
	In this section we show how $J$ distinguishes itself amongst all functors $K: T\text{-{\bf ASpa}}(\CC)\to\EE$ over $\CC$, with values in a category $\EE$ that live over $\CC$ via a (not necessarily faithful) functor $P:\EE\to\CC$ admitting {\em some} $P$-cartesian liftings. For clarity and fixation of our notation, we recall the relevant definitions.
	
	\begin{defi}\label{cartesiandefi}
	\rm
	Let $P:\EE\to \CC$ be a functor.	
	
	(1) A morphism $f:M\to N$ in $\EE$ is {\em $P$-cartesian} if, given any morphisms $g:L\to N$ in $\EE$ and $w:PL\to PM$ in $\CC$ with $Pf\cdot w=Pg$, there is a uniquely determined morphism $t:L\to M$ in $\EE$ with $Pt=w$ and $f\cdot t= g$. (Note that, when $P$ is faithful, the last equality and the uniqueness requirement for $t$ are consequences of the equality $Pt=w$.)
	
	(2) Given a morphism $u:X\to PN$ in $\CC$ with $N\in\EE$, a {\em $P$-cartesian lifting} of these data is a $P$-cartesian morphism $f:M\to N$ in $\EE$ with $PM=X$ and $Pf=u$. If for all such data there is a (chosen) $P$-cartesian lifting, then $P$ is a {\em (cloven) fibration}.
	\end{defi}
	
	\begin{rem}\label{topologicalfunctorrem} 
	\rm
	Every topological functor is a faithful fibration. In fact, topologicity of $P$ may be defined by replacing the single morphism $u:X\to PN$ above by a (possibly large) family $u_i:X\to PN_i\;(i\in I)$ and then requiring the existence of a jointly $P$-cartesian family $f_i:M\to N_i\;(i\in I)$ in $\EE$  with $Pf_i=u_i$ for all $i$, under the obvious extension of the notion of $P$-cartesianess. Faithfulness then becomes a consequence of this extended $P$-cartesian lifing property (see \cite{AHS}, \cite{MonTop} for details).		
	\end{rem}
	
	A routine exercise gives the following functorial version of the notion of $P
	$-cartesian lifting:
		
\begin{lm} For functors $E:\DD\to\EE,\;P:\EE\to \CC$ and $Q:\DD\to \CC$, let $\gamma:Q\to  PE$ be a natural transformation such that, for all objects $X\in \DD$, we are given a $P$-cartesian lifting $\delta_X:FX\to EX$ of $\gamma_X:QX\to PEX$. Then there is a unique way of extending the object assignment $X\mapsto FX$ to become a functor $F:\DD\to\EE$ with $PF=Q$, such that $\delta: F\to E$ is a natural transformation with $P\delta=\gamma$. 
\end{lm}

We refer to $\delta$ above as a {\em pointwise $P$-cartesian lifting} of $\gamma:Q\to PE$ and note that, when $P$ is a (cloven) fibration, every natural transformation $\gamma: Q\to PE$ admits a pointwise $P$-cartesian lifting.

Let us now look at the reflection $\beta: \mathrm{Id}_{T\text{-{\bf CSpa}}(\CC)}\to J\mathrm B$, {\em i.e.}, at the unit of the adjunction $\mathrm B\dashv J$.	By the definition of the category $T$-{\bf CSpa}$(\CC)$, every component of $\beta$ is $V$-cartesian, so that $\beta$ is trivially a pointwise $V$-cartesian lifting of $V\beta:V\to VJ\mathrm B$. We characterize $J$ as the universal functor over $\CC$ admitting a pointwise cartesian lifting of $V\beta$, as follows:

\begin{theorem}\label{maintheorem}
	Let $K:T\text{-{\bf ASpa}}(\CC)\to\EE$ and $P:\EE\to\CC$ be functors, with $PK=U=VJ$ being the forgetful functor, and assume that there is a pointwise $P$-cartesian lifting $\vartheta:\overline{K}\to K\mathrm B$ of $V\beta:V\to VJ\mathrm B=PK\mathrm B$. Then:
	\begin{itemize}
	\item[{\rm (1)}] The functor $\overline{K}:T\text{-{\bf CSpa}}(\CC)\to \EE$ satisfies the following properties {\rm (a),(b),(c)}; moreover, property {\rm (b)} determines $\overline{K}$ uniquely, while property {\rm(c)} still determines $\overline{K}$ up to a unique natural isomorphism whose $P$-image is an identity morphism:
	\begin{itemize}
	\item[{\rm (a)}]$P\overline{K}=V$;
	\item[{\rm (b)}] there is a natural isomorphism  $\iota: \overline{K}J\to K$ with $ P\iota=1_U$ and $\iota \mathrm B\cdot \overline{K}\beta=\vartheta$;
	\item[{\rm (c)}] $\overline{K}$ maps $\beta$ to a pointwise $P$-cartesian lifting of $V\beta$. 
	\end{itemize}
	\item[{\rm (2)}] Given any functor $H:  T\text{-{\bf CSpa}}(\CC)\to\EE$ with $PH=V$ and a natural transformation	$\kappa: HJ\to K$ with $P\kappa=1_U$, there is a unique natural transformation $\overline{\kappa}:H\to\overline{K}$ with $P\overline{\kappa}=1_V$ and $\iota\cdot \overline{\kappa}J=\kappa$.
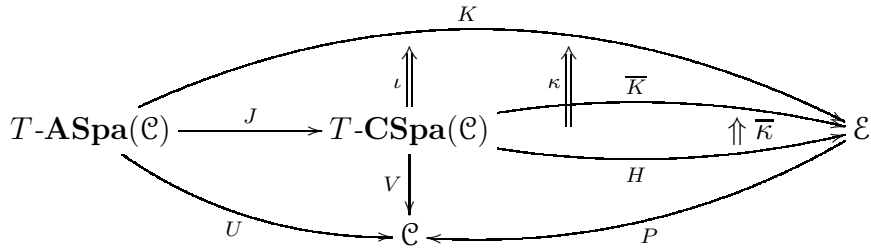
	\begin{center}
$\xymatrix{&&&&\\
T\text{-{\bf ASpa}}(\CC)\ar@/_1.0pc/[rrd]_U\ar[rr]^J\ar@/^3.2pc/[rrrrrr]^{ K\;\;\;} & & T\text{-{\bf CSpa}}(\CC)\ar@{=>}[u]^{\iota}\ar@/^0.9pc/[rrrr]^{\overline{K}}\ar@/_0.9pc/[rrrr]_H\ar[d]_V &\ar@{=>}[u]^{\kappa} && \Uparrow\overline{\kappa}& \EE\ar@/^1.2pc/[lllld]^P \\
& &\CC & &&\\
}$
\end{center}
		\end{itemize}
	\end{theorem}
	\begin{proof}
(1) The existence of the functor $\overline{K}$ is secured by the Lemma, and it satisfies (a) by its definition.	Since the functor $J$ is fully faithful, the counit $\varepsilon:\mathrm BJ\to\mathrm{Id}_{T\text{-{\bf Alg}}(\CC)}$ of the adjunction $B\dashv J$ is an isomorphism. Furthermore, also the transformation $\vartheta J$ is an isomorphism, as a P-cartesian lifting of the natural isomorphism $V\beta J$. Hence we obtain the composite  natural isomorphism
$$\xymatrix{\iota=(\overline{K}J\ar[rr]^{\vartheta J} && K\mathrm BJ\ar[rr]^{K\varepsilon} && K),
}$$
which satisfies $$P\iota= PK\varepsilon\cdot P\vartheta J=VJ\varepsilon\cdot V\beta J= 1_U,$$
by one of the triangular identities of the adjunction $\mathrm B\dashv J$. Furthermore,
$$ \iota \mathrm B\cdot \overline{K}\beta=K\varepsilon \mathrm B\cdot\vartheta J\mathrm B\cdot H\beta= K\varepsilon \mathrm B\cdot K\mathrm B\beta\vartheta=\vartheta,$$
by the other triangular identity of $\mathrm B\dashv J$.
This confirms that $\overline{K}$ satisfies (b), and (c) is actually an immediate consequence of (b) since, with $\iota \mathrm B$ being an isomorphism,  $\overline{K}\beta$ inherits its $P$-cartesianess from $\vartheta$, and $P\overline{K}\beta= P\iota \mathrm B\cdot P\overline{K}\beta=P\vartheta=V\beta$ holds. 
To show the uniqueness statement for $\overline{K}$, first let $L:T\text{-{\bf CSpa}}(\CC)\to\EE$ be a functor admitting a natural (iso)morphism $\lambda: LJ\to K$ and satisfying ($PL=V,\; P\lambda= 1_U$ and) $\lambda \mathrm B\cdot L\beta=\vartheta$. But the latter equality means in particular that
\begin{center}
$\xymatrix{L\ar[r]^{L\beta} & LJ\mathrm B\ar[r]^{\lambda \mathrm B} & K\mathrm B\\
}$
 \end{center}
has the same domain as $\vartheta$, so that $L=\overline{K}$. Assuming, instead, that $L\beta$ is  some $P$-cartesian lifting of $V\beta:V\to PK\mathrm B$ still means that the domain $L$ of $L\beta$ can differ from the domain $\overline{K}$ of the given $P$-cartesian lifting $\vartheta$ only by a unique isomorphism $L\to\overline{K}$ whose $P$-image is an identity morphism. (This argumentation is carried out more generally and in greater detail in (2).)

(2) With $H$ and $\kappa$ as given, for every $T$-space $X$ we obtain, by the $P$-cartesianess of $\vartheta_X$, the morphism $\overline{\kappa}_X:HX\to\overline{K}X$ in $\EE$, determined
by $P\overline{\kappa}_X=1_{VX}$ and $\vartheta_X\cdot \overline{\kappa}_X=\kappa_{\mathrm BX}\cdot H\beta_X$.

\begin{center}
$\xymatrix{& \overline{K}X\ar@{..}[ddd]\ar[rrr]^{\vartheta_X} &&& K\mathrm BX\ar@{..}[ddd] \\
HX\ar@{..}[ddr]\ar[ru]^{\overline{\kappa}_X}\ar[rrr]^{H\beta_X} &&& HJ\mathrm BX\ar@{..}[rdd]\ar[ru]_{\kappa_{\mathrm BX}} & \\
&&&&\\
& VX\ar[rrr]^{V\beta_X=PH\beta_X} &&& VJ\mathrm BX= PK\mathrm BX
}$	
\end{center}
 The naturality of $\vartheta, H\beta$ and $\kappa \mathrm B$ makes also $\overline{\kappa}: H\to\overline {K}$ natural, and we have $P\overline{\kappa}=!_V$ by definition, and $$\iota\cdot\overline{\kappa}J= K\varepsilon\cdot\vartheta J\cdot\overline{\kappa}J =K\varepsilon\cdot \kappa \mathrm BJ\cdot H\beta J= \kappa\cdot HJ\varepsilon\cdot H\beta J=\kappa$$
by adjunction. To see the uniqueness of $\overline{\kappa}$, let the natural transformation $\mu: H\to\overline{K}$ satisfy $P\mu=1_V$ and $\iota\cdot\mu J=\kappa$. One then has 
$$\kappa \mathrm B\cdot H\beta =\iota \mathrm B\cdot\mu J\mathrm B\cdot H\beta= \iota \mathrm B\cdot \overline{K}\beta\cdot\mu= \vartheta\cdot\mu,   $$
so that $\mu$ satisfies the defining conditions of $\overline{\kappa}$ and must therefore equal $\overline{\kappa}$.
	\end{proof}

 \begin{cor}[Burroni \cite{Burroni1971}]
 		Let $P:\EE\to C$ be a fibration. Then every functor $K:T\text{-{\bf ASpa}}(\CC)\to\EE$ with $PK=U$ admits a pseudo-extension $\overline{K}: T\text{-{\bf CSpa}}(\CC) \to \EE$ over $\CC$ (so that $\overline{K}	J\cong K$ and $P\overline{K}=V$) which, up to isomorphism, is uniquely determined by the property of mapping $\beta$ to a pointwise $P$-cartesian natural transformation.
 		  \end{cor}

\begin{rems}\label{finalrems}
\rm 
(1) The paper \cite{Burroni1971} actually claims	 that $\overline{K}$ transforms {\em all} $V$-cartesian monotone morphisms into $P$-cartesian morphisms---not just the $V$-cartesian morphisms $\beta_X\;(X\in T\text{-{\bf CSpa}}(\CC)$). We have not been able to confirm this claim and conjecture that it does not hold in general.

(2) Property (2) of the Theorem establishes that  $\overline{K}$ and $\iota$ form a {\em right Kan extension of $K$ along $J$} (see \cite{MacLane1971}), not in {\bf CAT}, but in the (huge) 2-category $\mathbf{CAT}/\CC$, which is defined by:
	\begin{itemize}
	\item the objects are functors $G:\AA\to\CC$;
	\item the morphisms $F: G\to H$ (with $H:\BB\to\CC$) are functors with $HF=G$;
	\item the 2-cells $\varphi:F\Longrightarrow F'$ are natural transformations with $H\varphi=1_G$.	
	\end{itemize}
	In fact, our proof shows that a significant part of Theorem \ref{maintheorem} may be established in any 2-category in lieu of {\bf CAT}.
	
	(3) It is important to note that the universal property described by the Theorem does not classify $T\text{-{\bf CSpa}}(\CC)$ as a topological extension of the monadic category  $T\text{-{\bf ASpa}}(\CC)$ over $\CC$. In fact, the paper \cite{Tholen1979} describes (in a more general context) a different extension category which is also topological over $\CC$ and contains $\CC^T\simeq T\text{-{\bf ASpa}}(\CC)$ as a full reflective subcategory. One simply considers the category 
	$$\mathsf{Gen}(\CC^T)$$
	of Eilenberg-Moore $T$-algebras $(R,r)$ that come with a specified system $X$ of generators. Hence, objects in $\mathsf{Gen}(\CC^T)$ are given by pairs $(p:X\to R, r)$ where $(R,r)$ is a $T$-algebra and $p:X\to R$ is a $\CC$-morphism, such that the $T$-homomorphism $p^{\sharp}=r\cdot Tp:(TX,\mu_X)\to (R,r)$ is a regular epimorphism in $\CC$ (and, hence, in $\CC^T$); a morphism $(f,f^*):(p,r)\to (q,s)$ is given by a $\CC$-morphism $f$ and a $T$-homomorphism $f^*:(R,r)\to(S,s)$ making the diagram
	\begin{center}$\xymatrix{X\ar[d]_p\ar[rr]^f && Y\ar[d]^q\\
	R\ar[rr]^{f^*} && S\\
	}$
	\end{center}
	commute. (We note that $f^*$ is determined by $f$, but we find it convenient to define $\mathsf{Gen}(\CC^T)$ rather explicitly as a full subcategory of the comma-category $\CC\downarrow U$, where we now consider the forgetful $U$ as a functor $\CC^T\to\CC$.) By an application of a more general result of \cite{Tholen1979} (whereby so-called semi-topological functors are precisely the restrictions of topological functors to full reflective subcategories), it follows that:
	\begin{itemize}
	\item	the ``domain functor'' $W:\mathsf{Gen}(\CC^T)\lra\CC,\; (p:X\to R,\,r)\longmapsto X,$ is topological and, in particular, a fibration, with the $W$-cartesian morphisms $(f,f^*)$ being characterized as those with $f^*$ monic in $\CC$ (and, hence, in $\CC^T$); 
	\item the ``codomain functor'' $C: \mathsf{Gen}(\CC^T)\lra\CC^T,\;(p:X\to R,\,r)\longmapsto(R,r)$ is the reflector of the full embedding $I:\CC^T\lra\mathsf{Gen}(\CC^T),\;(R,r)\longmapsto (1_R, r)$, with the adjunction units $\gamma_{(p,r)}:(p,r)\lra IC(p,r)$ given by the squares
	\begin{center}
	$\xymatrix{X\ar[d]_p\ar[rr]^p && R\ar[d]^{1_R}\\
	R\ar[rr]^{1_R} && R\\
	}$	
	\end{center}
	\end{itemize}
This establishes the diagram
\begin{center}
$\xymatrix{\CC^T\ar[rd]_U\ar@/_0.5pc/[rr]_I^{\bot} &  & \mathsf{Gen}(\CC^T)\ar[ld]^W\ar@/_0.5pc/[ll]_{C}\\
& \CC & \\
}$
	\end{center}
for which one may prove exactly the same statements that we have established in Theorem \ref{maintheorem} for the initial diagram of Section 5 --- one just has to write $I,C,\gamma $ for $J, V, \beta$, respectively, as indicated by the following diagram.
	\begin{center}
$\xymatrix{&&&&\\
T\text{-}\mathbf{ASpa}(\CC)\simeq\CC^T\ar@/_1.0pc/[rrd]_U\ar[rr]^I\ar@/^3.2pc/[rrrrrr]^{ K\;\;\;} & & \mathsf{Gen}(\CC^T)\ar@{=>}[u]^{\iota}\ar@/^0.9pc/[rrrr]^{\overline{K}}\ar@/_0.9pc/[rrrr]_H\ar[d]_W &\ar@{=>}[u]^{\kappa} && \Uparrow\overline{\kappa}& \EE\ar@/^1.2pc/[lllld]^P \\
& &\CC & &&\\
}$
\end{center}
Then, an application with $K:= J: T\text{-}\mathbf{ASpa}(\CC)\lra T\text{-}\mathbf{CSpa}(\CC)$ produces the functor $\overline{J}: \mathsf{Gen}(\CC^T)\lra 	T\text{-}\mathbf{CSpa}(\CC)$ with $\overline{J} I\cong J$, whereas, when putting $K:=I$ in Theorem	\ref{maintheorem}, we obtain a functor $\overline{I}:	T\text{-}\mathbf{CSpa}(\CC)\lra \mathsf{Gen}(\CC^T)$ with $\overline{I}J\cong I$.
Rather than pursuing these functors further in generality, in the following Example we present them explicitly in our role model $T=\UU$ (and $\CC=\Set$) and show that $\overline{I}$ is a full coreflective embedding with right adjoint $\overline{J}$, but fails to be an equivalence of categories.
\end{rems}

\begin{exa}
\rm 
	The objects of the category $\mathsf{Gen}(\mathbf{KHaus})$ are given by $\Set$-maps $p:X\to R$, with $R$ carrying a compact Hausdorff topology, such that every point $r\in R$ is the limit point of $p[\mathfrak x]$, for some ultrafilter $\mathfrak x$ on $X$ (where, for simplicity, we have written $p[\mathfrak x]$ for $\UU p(\mathfrak x)$). A morphism $(p:X\to R)\lra(q:Y\to S)$ is a pair $(f:X\to Y,f^*:R\to S)$ of maps with $f^*p=qf$, such that $f^*$ is continuous (and, hence, determined by $f$, via the formula $f^*(r)=\lim q[f[\mathfrak x]]$ whenever $r=\lim p[\mathfrak x]$). 
	
	The category $\mathbf{KHaus}$ gets fully embedded into $\mathsf{Gen}(\mathbf{KHaus})$ via $I: R\mapsto 1_R$. Up to isomorphism, it extends to the full embedding $\overline{I}: \mathbf{CReg}\lra\mathsf{Gen}(\mathbf{KHaus})$, which maps the completely regular space $X$ to the map 
	$\tilde{\beta}_X:X\to\tilde{\mathrm B}X$, where $\tilde{\mathrm B}X$ is the (closed) image of the map $\beta_X^{\sharp}:\UU X\to\mathrm BX, \mathfrak x\mapsto \lim \beta_X[\mathfrak x]$, and where $\tilde{\beta}_X$ is the restriction of $\beta_X:X\to \mathrm BX$. On the other hand, the inclusion functor $J:\mathbf{KHaus}\hookrightarrow\mathbf{CReg}$ gets extended along $I$ to the functor $\overline{J}:\mathsf{Gen}(\mathbf{KHaus})\lra \mathbf{CReg}$, which simply sends an object $p:X\to R$ to the set $X$, provided with the initial topology with respect to $p$ and the compact Hausdorff space $R$.
	Briefly, in the following diagram, the arrows commute in both directions, up to isomorphism:
	\begin{center}
	$\xymatrix{ & \mathbf{KHaus}\ar@{>->}[ld]_J\ar@{>->}[rd]^I &\\
	\mathbf{CReg}\;\ar@{>->}[rr]^{\overline{I}}_{\bot} && \mathsf{Gen}(\mathbf{KHaus})\ar@/^1.0pc/[ll]^{\overline{J}}
	}$	
	\end{center}
Furthermore, $\overline{J}$ is right adjoint to $\overline{I}$ which, therefore, is a coreflective embedding; indeed, one easily sees that, for a completely regular space $Z$, continuous maps $f: Z\to X$, where $X$ carries the initial topology with respect to $p:X\to R,\; R$ compact Hausdorff,  correspond bijectively (and naturally) to morphisms $(f,f^*):\overline{I}Z\lra p$ in $\mathsf{Gen}(\mathbf{KHaus})$. Note also that $\overline{J}I\cong J$  implies $C\overline{I}\cong\mathrm B$, where $C$ is the left adjoint of $I$, sending $(p:X\to R)$ to $R$. 

Finally, $\overline{I}$ (as well as $\overline{J}$) fails to be an equivalence of categories, since $\overline{I}\overline{J}$ fails to be isomorphic to $\mathrm{Id}_{\mathsf{Gen}(\mathbf{KHaus})}$; for example, it maps the embedding of the discrete $\mathbb N$ into its one-point compactification to the $\check{\mathrm C}$ech-Stone compactification $\beta_{\mathbb N}:\mathbb N\to\mathrm B\mathbb N$. 
\end{exa}
	
\section{Some remarks on related work}

For $\CC=\Set$, the structure of a $T$-space $X$ as presented in Section 3 is given by a relation 
$C\subseteq TX\times X$; equivalently, by a map $TX\times X\lra \mathsf 2=\{\text{false, true}\}$. Monoidal topology \cite{MonTop}, as initiated in the papers \cite{CH2003}, \cite{CT2003}, \cite{CHT2004}, \cite{Seal2005}, extends Barr's relational $T$-algebras and, hence, Burroni's $T$-orders at the level of sets, in two ways. On one hand, one replaces the Boolean algebra $\mathsf 2$ by a quantale $\mathsf V$, {\em i.e.}, by a complete lattice $\mathsf V$ that comes with a monoid structure whose multiplication preserves arbitrary joins in each variable. This permits in particular the inclusion of metrically enriched structures into the setting (as initiated by Lawvere \cite{Lawvere1973}). Following Lowen's \cite{Lowen1997} and other authors' works, these have enjoyed much attention in recent years. 

On the other hand, for the transitivity axiom (T) of 
Definition \ref{definitionUAlg}, the monoidal topology setting allows for considerable flexibility of how the relation $C$ gets extended to a relation $\hat{C}\subseteq TTX\times TX$; instead of prescribing the so-called Barr-extension of the monad $T$, one considers a fairly general lax-functorial extension of $T$ from mappings to $\mathsf V$-valued relations of sets.
Among many other applications, this additional flexibility led to the unexpected result \cite{MonTop} that the same two-axiom presentation of topological spaces in terms of ultrafilter convergence remains valid if one replaces ultrafilters by filters, under an appropriate lax extension of the filter monad.

Briefly then, the setting of monoidal topology extends the Barr-Burroni setting considerably, in the case $\CC=\Set$. However, it does not allow for an easy extension to other base categories. Put in more categorical terms, monoidal topology leans on enriched categorical methods (with the quantale $\mathsf V$ being treated as a ``thin'' monoidal-closed category), while the general Burroni setting extends the theory of small categories and preorders internal to a category $\CC$, using an arbitrary monad on $\CC$.

A large portion of Ralph Kopperman's work, pursued over several decades, concerns generalized metric structures and their applications in topology and computer science; see, for example, \cite{Kopperman1988}, \cite{FlaggKopper1997}, \cite{BuKoMa2009}. It is therefore closely related to the  monoidal-topology setting, but there is a significant conceptual difference, as he normally considers the ``value recipient'' $\mathsf V$ to vary with, and be part of, the structure of the space-like objects that are under investigation. 

\end{document}